\documentclass[12pt,english]{article}

\usepackage[utf8]{inputenc}

\usepackage{hyperref}

\usepackage{geometry}
\usepackage{color}
\usepackage{amsmath}
\usepackage{amssymb}
\usepackage{amsthm}

\def\E{\hskip.15ex\mathrm{E}\hskip.10ex}
\def\P{\mathrm{P}}

\def\phi{\varphi}



\newtheorem{theorem}{Theorem}
\newtheorem{Theorem}[theorem] {Theorem}
\newtheorem{Lemma}[theorem] {Lemma}
\newtheorem{Proposition}[theorem]{Proposition}
\newtheorem{Definition}[theorem]{Definition}
\newtheorem{Remark}[theorem] {Remark}
\newtheorem{Corollary}[theorem] {Corollary}
\newtheorem{Ex}[theorem]{Example}

\begin{document}
\global\long\def\E{\mathbb{E}}
\global\long\def\P{\mathbb{P}}
\global\long\def\N{\mathbb{N}}
\global\long\def\ind{\mathbb{I}}

\title{
{\normalsize\tt\hfill\jobname.tex}\\
On improved convergence conditions and bounds for Markov chains}
\author{A.Yu. Veretennikov\footnote{University of Leeds, UK; National Research University Higher School of Economics, and Institute for Information Transmission Problems, Moscow, Russia, email: ayv @ iitp.ru. The sections 1 and 2 of the article were prepared within the framework of the HSE University Basic Research Program; 
the section 3 was supported by Russian Foundation for Basic Research grant 20-01-00575a; in section 4 this research was supported in part through computational resources of HPC facilities at NRU HSE.}, 
M.A. Veretennikova\footnote{University of Warwick, UK, email: maveretenn @ gmail.com}
}

\maketitle
\begin{abstract}
Improved rates of convergence for ergodic Markov chains and relaxed conditions for them, as well as analogous convergence results for non-homogeneous  Markov chains are studied. The setting 
from the previous works is extended.  Examples are provided where the new bounds are better and where they give the same convergence rate as in the classical Markov -- Dobrushin inequality.

~

\noindent
{\bf  Keywords:}
Markov chains;  ergodicity; extension of Markov -- Dobrushin condition; convergence rate
\end{abstract}

\section{Introduction}\label{sec:intro}
It is well-known that for a discrete irreducible aperiodic homogeneous Markov chain $(X_n, n\ge 0)$ with a finite state space $S$ there is a unique stationary distribution $\mu$ towards which the distribution of the chain $\mu_n={\cal L}(X_n)$ converges exponentially fast uniformly with respect to the initial distribution $\mu_0$:
\begin{equation}\label{finiteMD}
\|\mu_n - \mu\|_{TV} \le 2 (1-\kappa)^n, 
\end{equation}
where 
\begin{equation*}
\kappa = \min_{i,i'}\sum_{j\in {\cal S}} p_{ij}\wedge p_{i'j}, 
\end{equation*}
and where $p_{ij}$ are the transition probabilities, see, e.g., \cite[Theorem of section 17]{Gnedenko}, or \cite[Section 4]{Kolm} (strictly speaking, in both classical sources only convergence is shown, but an exponential bound does follow from the proofs). A similar bound exists for general Markov chains, too (cf., among many other sources, \cite{Veretennikov17}); here $\|\cdot\|_{TV}$ is the distance of total variation between measures. Concerning the history of 
Markov -- Dobrushin's bounds (MD bounds in what follows) 
for homogeneous chains see 
\cite{Doeblin, Doob53, 
Kolm, 
Markov, Markov2,  
Seneta1, Seneta3}, 
et al.  Another well-known method of exponential estimates relates to eigenvalues of the transition probability matrix \cite[Chapter XIII, formula (96)]{Gantmacher}, \cite[Theorem 1.2]{Seneta1}. However, it is less general, being mainly restricted to the situations of finite state spaces, especially $|{\cal S}| = 2,3$ (the latter case can be found in various textbooks) and to the reversible Markov chains, see \cite{Stroock-Diaconis}. When the latter method is applicable, it provides the best results. Yet, generally speaking, it does not work for the non-homogeneous chains, unlike the classical MD approach, the new extension of which is presented in this paper. 

Note that for the non-homogeneous MC a bound similar to (\ref{finiteMD}) holds true, 
\begin{equation}\label{finiteMDn}
\|\mu_n - \mu'_n\|_{TV} \le 2 \prod_{t=0}^{n-1}(1-\kappa_t), 
\end{equation}
where 
\begin{equation*}
\kappa_t = \min_{i,i'}\sum_{j\in {\cal S}} p_{ij}(t)\wedge p_{i'j}(t), 
\end{equation*}
where $p_{ij}(t)$ are the transition probabilities at time $t$ and $\mu_n$ and $\mu'_n$ are distributions of the same MC with any two different initial distributions  $\mu_0$ and $\mu'_0$; of course, whether or not (\ref{finiteMDn}) implies convergence depends on the divergence of the series $\sum_{t=1}^{\infty}\kappa_t$. The bound (\ref{finiteMDn}) follows from the calculus analogous to the one for the homogeneous case as in \cite{Gnedenko}; see \cite[Theorem 4.III, inequality (29)]{Kolm}.

Provisional variants of the new bound for homogeneous MC based on versions of Markov coupling were proposed in papers \cite{BV} and \cite{Veretennikov17} (see \cite{Griffeath, Lindvall, Thorisson, Vaserstein} about coupling). In the present paper it turned out to be possible to relax the requirement of the unique dominating measure for all transition kernels,  and to do it in a more general state space (in \cite{BV} the state space was finite-dimensional Euclidean). The non-homogeneous case is treated in a separate section despite many similarities to the homogeneous one: the main reason is that in this case there is no invariant measure.  Finally, in the last section of the paper quite a few examples are presented (in \cite{BV} and \cite{Veretennikov17} there were no examples at all, while \cite{VV2020} provides four), which show that in the majority of cases the classical MD bound (\ref{finiteMD}) may be, indeed, improved effectively.

For any process $(X_n, \, n\ge 0)$ denote
$$
{\cal F}^X_n = \sigma(X_k: \, k\le n); \quad {\cal F}^X_{(n)} = \sigma(X_n). 
$$ 
Also, the following notations from the theory of Markov processes will be accepted (cf. \cite{Dynkin}): the index $x$ in $\mathbb E_x$ or $\mathbb P_x$ signifies the expectation, or, respectively, the probability measure related to the non-random initial state of the process $X_0=x$. This initial state may be also random with some distribution $\mu$, in which case notations  $\mathbb E_\mu$ and $\mathbb P_\mu$ may be used.

The paper consists of four sections. Section \ref{sec:intro} is this introduction. Section \ref{sec:homo} contains the presentation of the version of Markov coupling for the homogeneous case, the operator approach  and the main theorem \ref{lastthm2} for this case; most proofs of lemmata in this section are dropped because they are special cases of their non-homogeneous analogues proved in the next section.  Section \ref{sec:nonh} treats results for the non-homogeneous case. Section \ref{sec:ex} offers examples where the new bound (for the homogeneous case) is better and where it gives asymptotically the same result as in the classical MD inequality. Both bounds are also compared to the estimate provided by the eigenvalue method.

\section{Markov coupling, homogeneous case}\label{sec:homo} 
We firstly consider a homogeneous Markov process (MP) in discrete time $(X_n, \, n\ge 0)$ on a general (nonempty) state space $S$ with a topology and with a Borel sigma-algebra $\sigma(S)$; as usual in Markov processes, any state $\{x\}$ belongs to $\sigma(S)$.  
If the state space $S$ is finite, then $|S|$ denotes the number of its elements and $\cal P$ stands for the transition matrix $\left(p_{ij}\right)_{1\le i,j \le |S|}$ of the process. Such a notation may also be applied in the case where $S$ is countable.

All proofs of the lemmata 
of this section except for the proof of lemma \ref{odvuh} are postponed till the next section where their analogues for the non-homogeneous cases will be established; they include homogeneous situations, too.

In the well-known inequality (\ref{exp_bd32}) of the Proposition \ref{thm_erg21}  stated below for the reader's convenience (without proof), which generalises the bound (\ref{finiteMD}), it is assumed that the Markov--Dobrushin's constant is positive:
\begin{equation}\label{MD}
 \kappa := \inf_{x,x'} \int \left(\frac{\P_{x'}(1,dy)}{\P_x(1,dy)}\wedge 1 \right)\P_x(1,dy) > 0. 
\end{equation}
A similar coefficient and condition can be introduced for any number of steps $m\ge 1$:
\begin{equation}\label{MDm}
 \kappa^{(m)} := \inf_{x,x'} \int \left(\frac{\P_{x'}(m,dy)}{\P_x(m,dy)}\wedge 1 \right)\P_x(m,dy) >0. 
\end{equation}
In the main result of this section  -- Theorem \ref{lastthm2} -- this condition will be relaxed. 
Note that here $\displaystyle \frac{\P_{x'}(1,dy)}{\P_x(1,dy)}$ is understood in the sense of the density of the absolute continuous component of the numerator with respect to the denominator measure. For brevity we will be using a simplified notation $\P_x(dz)$ for $\P_x(1,dz)$. Note that for any Borel measurable $A$, the function  $\P_x(A)$ is Borel measurable with respect to $x$ which is a standard requirement in Markov processes \cite{Dynkin}. Due to linearity such measurability with respect to the pair $x,x'$ will be also valid for the measure $\Lambda_{x,x'}$ defined below. 
For two fixed states $x, x'$ denote 
\begin{equation*}
\Lambda_{x,x'}(dz) := (\P_x(dz) + \P_{x'}(dz))/2.
\end{equation*} 
Likewise, for any $m\ge 1$, let 
\begin{equation*}
\Lambda^{(m)}_{x,x'}(dz) := (\P_x(m,dz) + \P_{x'}(m,dz))/2.
\end{equation*} 
Note that $\Lambda_{x,x'}(dz) = \Lambda_{x',x}(dz)$, and $\Lambda^{(m)}_{x,x'}(dz) = \Lambda^{(m)}_{x',x}(dz)$.
\begin{Lemma}\label{newMD2}
The following representation for the condition (\ref{MD}) holds true, 
\begin{equation}\label{betterMD2}
 \kappa = \inf_{x,x'} \int \left(\frac{\P_{x'}(dy)}{\Lambda_{x,x'}(dy)}\wedge \frac{\P_{x}(dy)}{\Lambda_{x,x'}(dy)} \right)\Lambda_{x,x'} (dy)>0. 
\end{equation}
The same is valid for $\kappa^{(m)}$ for any $m$,
$$
 \kappa^{(m)} = \inf_{x,x'} \int \left(\frac{\P_{x'}(m,dy)}{\Lambda^{(m)}_{x,x'}(dy)}\wedge \frac{\P_{x}(m,dy)}{\Lambda^{(m)}_{x,x'}(dy)} \right)\Lambda^{(m)}_{x,x'} (dy)>0.
$$

\end{Lemma}

\begin{Remark}
Note that the right hand side in (\ref{betterMD2}), actually, does not depend on any particular reference measure $\Lambda_{x,x'}$ (even if it is not symmetric with respect to $x,x'$), i.e., for any other measure with respect to which both $\P_{x'}(dy)$ and $\P_{x}(dy)$ are absolutely continuous the formula (\ref{betterMD2}) gives the same result. Indeed, it follows straightforwardly from the fact that if $d\Lambda_{x,x'} <\!\!< d\tilde\Lambda_{x,x'}$ and $d\Lambda_{x,x'} = \phi_{x,x'} d\tilde\Lambda_{x,x'}$, then we get, 
\begin{eqnarray*}
&\displaystyle  \int \left(\frac{\P_{x'}(dy)}{\Lambda_{x,x'}(dy)}\wedge \frac{\P_{x}(dy)}{\Lambda_{x,x'}(dy)} \right)\Lambda_{x,x'} (dy) 
 \\\\
&\displaystyle  = \int \left(\frac{\P_{x'}(dy)}{\phi_{x,x'}\tilde\Lambda_{x,x'}(dy)}\wedge \frac{\P_{x}(dy)}{\phi_{x,x'}(y)\tilde\Lambda_{x,x'}(dy)} \right)\phi_{x,x'}(y) 1(\phi_{x,x'}(y)>0)\tilde\Lambda_{x,x'} (dy) 
 \\\\
&\displaystyle  = \int \left(\frac{\P_{x'}(dy)}{\tilde\Lambda_{x,x'}(dy)}\wedge \frac{\P_{x}(dy)}{\tilde\Lambda_{x,x'}(dy)} \right)1(\phi_{x,x'}(y)>0)\tilde\Lambda_{x,x'} (dy).
\end{eqnarray*}
However, $\P_{x'}(1,dy)<\!\!< \Lambda_{x,x'}(dy) = \phi_{x,x'}(y) \tilde \Lambda_{x,x'}(dy)$, so for any measurable $A$ we have $\int_A \P_{x'}(dy) 1(\phi_{x,x'}(y)=0) = 0$ and the same for $\P_x(dy)$, which means that, actually, 
\[
\int \left(\frac{\P_{x'}(dy)}{\tilde\Lambda_{x,x'}(dy)}\wedge \frac{\P_{x}(dy)}{\tilde\Lambda_{x,x'}(dy)} \right)1(\phi_{x,x'}(y)>0)\tilde\Lambda_{x,x'} (dy) = \int \left(\frac{\P_{x'}(dy)}{\tilde\Lambda_{x,x'}(dy)}\wedge \frac{\P_{x}(dy)}{\tilde\Lambda_{x,x'}(dy)} \right)\tilde\Lambda_{x,x'} (dy).
\]
Respectively, if there are two reference measures $\Lambda_{x,x'}$ and, say, $\Lambda'_{x,x'}$, then we may take $\tilde\Lambda_{x,x'} = \Lambda_{x,x'} + \Lambda'_{x,x'}$, and the coefficients computed by using each of the two -- $\Lambda_{x,x'}$ and $\Lambda'_{x,x'}$ -- will be represented via $\tilde\Lambda_{x,x'}$ in the same way. 
\end{Remark}

~

Here is the key notion in the following presentation in this section: denote 
\[
\kappa(x,x') : = \int \left(\frac{\P_{x'}(dy)}{\P_x(dy)}\wedge 1 \right)\P_x(dy).
\]
Also, let 
\[
\kappa^{(m)}(x,x') : = \int \left(\frac{\P_{x'}(m,dy)}{\P_x(m,dy)}\wedge 1 \right)\P_x(m,dy).
\]
Clearly, for any $x,x'\in S$, 
\begin{equation*}
\kappa(x,x') \ge \kappa, \quad 
\kappa^{(m)}(x,x') \ge \kappa^{(m)}. 
\end{equation*}

\begin{Lemma}\label{jj2}
For any $x,x'\in S$, and for any $m\ge 1$
\[
\kappa^{(m)}(x,x') = \kappa^{(m)}(x',x).
\]
\end{Lemma}
\noindent

\begin{Definition}
If an MC $(X_n)$ satisfies the condition (\ref{MD}), or (\ref{MDm}) with any $m\ge 1$ -- we call it the {\bf MD-condition} or {\bf MDm-condition}, respectively, in the sequel -- then we call this process {\em Markov--Dobrushin's or the {\bf MD-process}}.
\end{Definition}

This condition in a simple scenario for finite chains was introduced by Markov himself \cite[Section 5]{Markov}; later on, for non-homogeneous Markov processes its analogue was suggested and used by Dobrushin \cite{Dobrushin}. So, we call it Markov--Dobrushin's condition, as already suggested earlier by E. Seneta. Note that in all cases $\kappa \le 1$, and $\kappa^{(m)} \le 1$. The case $\kappa = 1$ corresponds to the i.i.d. sequence $(X_n)$. In the opposite extreme situation where the transition kernels are singular for different $x$ and $x'$, we have $\kappa = 0$. 
The MD-condition (\ref{MD}),  as well as (\ref{MDm}),  are most useful because they both provide {\bf effective} quantitative upper bounds for the convergence rate of a Markov chain towards its (unique) invariant measure in the total variation metric. The following classical result is provided for comparison: the bound (\ref{exp_bd32}) can be found in most textbooks on ergodic Markov chains; the bound (\ref{exp_bd42}) is an easy generalisation, also well-known, and (\ref{exp_bd52}) is just another version of (\ref{exp_bd42}).

\begin{Proposition}\label{thm_erg21}
Let the assumption (\ref{betterMD2}) 
hold true. Then the process $(X_n)$ is ergodic, i.e., there exists  a limiting probability measure $\mu$, which is stationary and such that the uniform bound is satisfied for every $n$, 
\begin{equation}\label{exp_bd32}
 \sup_{x}\sup_{A\in S} |\P_x(n,A) - \mu(A)| \le (1-\kappa)^{n}.
\end{equation}
Also, for any $m\ge 1$
\begin{equation}\label{exp_bd42}
\sup_{x}\sup_{A\in S} |\mu^x_n(A) - \mu(A)| \le  (1-\kappa^{(m)})^{[n/m]}, 
\end{equation}
and 
\begin{equation}\label{exp_bd52}
\sup_{x}\|\mu^x_n - \mu\|_{TV} \le 2 (1-\kappa^{(m)})^{[n/m]} (1-\kappa)^{n-m[n/m]}. 
\end{equation}

\end{Proposition}
Clearly, if the assumption (\ref{MD}) fails, the estimate (\ref{exp_bd32}) is still valid, but does not contain any information since the difference between two probabilities cannot exceed one in any case. Similarly, for  (\ref{exp_bd42}) and  (\ref{exp_bd52}) to make some  sense it is required that $\kappa^{(m)} > 0$, although, without this condition both inequalities are still valid. There are natural examples where the rate provided by (\ref{exp_bd42}) can be considerably better than (\ref{exp_bd32}): for example, it is just possible that $\kappa=0$, while $\kappa^{(2)}>0$.

The following  important folklore lemma answers the following question: suppose we have two distributions, which are not singular, and the ``common area'' under the two densities equals some positive constant $q$. Is it possible to realise these two distributions on the same probability space so that the two corresponding random variables {\em coincide} exactly with probability $q$? (Let us emphasize that the authors of this paper are not the authors of this lemma, but it is unknown to us where it was first published.)

\subsection{Coupling lemma}

\begin{Lemma}[``Of two random variables'']\label{odvuh}
Let $X^{1}$ and $X^2$ be two random variables on their 
probability spaces $(\Omega^1, {\cal F}^1, \mathbb P^1)$ and $(\Omega^2, {\cal F}^2, \mathbb P^2)$ and with densities $p^1$ and $p^2$ with respect to some reference measure $\Lambda$, correspondingly.  Then, if 
\begin{equation*}
q := \int \left(p^1(x)\wedge p^2(x)\right) \Lambda(dx) > 0, 
\end{equation*}
then there exists a probability space $(\Omega, {\cal F}, \mathbb P)$ and two random variables on it $\tilde X^1, \tilde X^2$, such that 
\begin{equation*}
{\cal L}(\tilde X^j) ={\cal L}(X^j), \; j=1,2, \quad \& \quad \mathbb  P(\tilde X^1 = \tilde X^2) = q. 
\end{equation*}
\end{Lemma}
See, for example,  \cite{Veretennikov17}; as it was said, this lemma will be used in the sequel. 
We briefly remind the proof because its short calculus will be needed in the proof of the next Lemma.

\medskip

\noindent
{\em Proof of the Lemma \ref{odvuh}.} 
{\bf 1. Construction.} We will now need {\em four} new independent random variables: a Bernoulli random variable  
$\zeta$ with $\mathbb P(\zeta=0) = q$ and $\eta^{1,2}$  and $\xi$ with the densities with respect to the measure $\Lambda$, respectively, 
\begin{eqnarray}\label{laws}
& \displaystyle p^{\eta^1}(x) := \frac{p^1 - p^1\wedge p^2}{\displaystyle\int (p^1 - p^1\wedge p^2)(y)\Lambda(dy)}(x), \quad
p^{\eta^2}(x) := \frac{p^2 - p^1\wedge p^2}{\displaystyle\int (p^2 - p^1\wedge p^2)(y)\Lambda(dy)}(x),
 \nonumber \\ \\ \nonumber 
& \displaystyle p^{\xi}(x) := \frac{ p^1\wedge p^2}{\displaystyle\int (p^1\wedge p^2)(y)\Lambda(dy)}(x),
\end{eqnarray}
where in the last expression it is assumed that the denominator is strictly  positive; the alternative case will be explained in the end of the proof; in the first two expressions it is also assumed that the denominator is strictly positive, and the alternative will be treated in the last step of the proof. 

We may assume that they are all defined on their own probability spaces and eventually we consider the {\bf direct product} of these probability spaces denoted as $(\Omega, {\cal F}, \mathbb P)$. As a result, they are all defined on one unique probability space and they are independent there. 
Now, on the same product of all  probability spaces just mentioned, let us define the random variables 
\begin{equation}\label{ety}
\tilde X^1:=  \eta^1 1(\zeta \not=0) + \xi 1(\zeta=0) , \quad \& 
\quad \tilde X^2:=  \eta^2 1(\zeta \not=0) +\xi 1(\zeta=0).
\end{equation}

~

\noindent
{\bf 2. Verification.} From (\ref{ety}), clearly, 
\[
\mathbb P(\tilde X^1=\tilde X^2) \ge  \mathbb P(\zeta=0) = q. 
\]
Yet, if $q<1$ then the distributions of $\eta^1$ and $\eta^2$ are singular, so in fact, we have an equality
\[
\mathbb P(\tilde X^1=\tilde X^2) = q. 
\]
If $q=1$, then we have, 
\[
\mathbb P(\tilde X^1=\tilde X^2) = q = 1. 
\]

~

\noindent
Next, since $\zeta$, $\xi$ and $\eta^{1}$ are independent on $(\Omega, \cal F, \mathbb P)$, then for any bounded measurable function $g$ we have, 
\begin{eqnarray*}
&\displaystyle \mathbb E g(\tilde X^1) 
= \mathbb E g(\tilde X^1)1(\zeta=0) + \mathbb E g(\tilde X^1)1(\zeta \not=0) 
 \\\\
&\displaystyle  
= \mathbb E g(\xi)1(\zeta=0) + \mathbb E g(\eta^1)1(\zeta \not=0) = \mathbb E g(\xi) \mathbb E 1(\zeta=0) + \mathbb E g(\eta^1) \mathbb E 1(\zeta \not=0) 
 \\\\
&\displaystyle = q \int g(y) p^{\xi }(y)\,\Lambda(dy)  
+(1- q) \int g(y) p^{\eta^1}(y)\,\Lambda(dy) 
 \\\\
&\displaystyle = q \int g(x) \frac{p^1\wedge p^2}{\displaystyle \int (p^1\wedge p^2)\Lambda(dy)}(x) \Lambda(dx) + (1- q) \int g(x) \frac{p^1 - p^1\wedge p^2}{\displaystyle \int (p^1 - p^1\wedge p^2)(y)\Lambda(dy)}(x)\Lambda(dx)
 \\\\
&\displaystyle =  \int g(x) p^1\wedge p^2 (x) \Lambda(dx) +  \int g(x) (p^1 - p^1\wedge p^2) (x)\Lambda(dx)
= \int g(y) p^1(y)\,dy = \mathbb E g(X^1).  
\end{eqnarray*}
For $\tilde X^2$ the arguments are similar, so also $\mathbb E g(\tilde X^2) = \mathbb E g( X^2)$.

\medskip

\noindent
{\bf 3.}
In the considerations above it was assumed that all the denominators are strictly positive. If any of them equals zero, the claim of the Lemma remains valid and becomes trivial. Yet, for the sequel it makes sense to re-define all four random variables in such cases, too. 

In the case $q=1$, clearly, $p^1 = p^2$. Let 
$$
p^{\eta^1}(x)= p^{\eta^2}(x) = p^{\xi}(x);  
$$
the definition of $\zeta$ does not change, but this random variable is then just a constant $\zeta = 0$ almost surely. The result is that the distributions of $X^1$ and $X^2$ coincide, so the formula (\ref{ety}) above can be implemented. 

In the case $q=0$, the only change needed is for $p^\xi$, because the denominator in the definition of this density equals zero. In fact, $p^\xi$ here can be defined arbitrarily and it would not change the result because the two distributions are singular with respect to each other. For the definiteness, we propose $p^\xi = p^1$ (however in the application of this lemma in the next subsection it will be re-defined, which will not change the conclusion). The same formula (\ref{ety}) can be used; yet, coupling is impossible, which is in agreement with the fact that $q=0$. 
The Lemma \ref{odvuh} is proved.

\subsection{Markov coupling (homogeneous) }\label{Sec22}
In this subsection it is explained how to apply the general coupling method to Markov chains in general state spaces $(S, {\cal S})$. Various presentations of this method may be found in \cite{Kalash, Lindvall, Nummelin, Thorisson, Vaserstein},  et al. This section follows the lines from \cite{BV}, which, in turn, is based on \cite{Vaserstein}. Note that in \cite{BV} the state space was $\mathbb R^1$; however, in $\mathbb R^d$ all formulae remain the same and this may be further extended to more general state spaces. 

Let us generalize the Lemma \ref{odvuh} to a sequence of random variables and present our coupling
construction for Markov chains based on \cite{Vaserstein}. Consider two versions $(X^1_n), (X^2_n)$ of the same Markov process with two initial distributions $\mu_0^1$ and  $\mu_0^2$ respectively (this does not exclude the case of non-random initial states). Denote 
\begin{equation*}
\kappa(0) :=
\int \left(\frac{\mu_0^1(dy)}{\mu_0^2(dy)}\wedge 1 \right)\mu_0^2(dy).
\end{equation*}
It is clear that $0\le \kappa(0)\le1$ similarly to $\kappa(u,v)$ for all $u,v$.  
We assume that $X^1_0$ and $X^2_0$ have different distributions, so $\kappa(0)<1$. Otherwise, we obviously have
$X^1_n\stackrel{d}{=}X^2_n$ (equality in distribution) for all $n$, and the coupling can be made trivially, for example, by letting  $\widetilde X^1_n= \widetilde
X^2_n:=X^1_n$.

Let us introduce a new, vector-valued homogeneous {\bf Markov process} $\left(\eta^1_n,\eta^2_n,\xi_n,\zeta_n\right)$. 
The values $\left(\eta^1_0,\eta^2_0,\xi_0,\zeta_0\right)$ are chosen directly on the basis of the Lemma \ref{odvuh} as $\left(\eta^1,\eta^2,\xi,\zeta\right)$, according to the distributions in (\ref{laws}).
In particular, if $\kappa_0=0$ then we can set
\begin{equation*}
\eta^1_0:=X^1_0,\; \eta^2_0:=X^2_0,\; \xi_0:=X^1_0,\; \zeta_0:=1.
\end{equation*}
(The value for $\xi_0$ is not important in this case.)
If $\kappa_0=1$ then we can set
\begin{equation*}
\eta^1_0:=X^1_0,\; \eta^2_0:=X^1_0,\; \xi_0:=X^1_0,\; \zeta_0:=0.
\end{equation*}

Now, by induction, assuming that the random variables $\left(\eta^1_n,\eta^2_n,\xi_n,\zeta_n\right)$ have been determined for some $n$, let us show how to construct them for $n+1$. For this aim, we define the transition probability density $\phi$ with respect to the same measure $\Lambda_{x^1, x^2}$ (in fact, $\Lambda_{x^1, x^2} \times \Lambda_{x^1, x^2} \times \Lambda_{x^1, x^2}\times (\delta_0 + \delta_1)/2$) for this (vector-valued) process as follows,
\begin{equation}\label{process_eta}
\phi(x,y):=\phi_1(x,y^1)\phi_2(x,y^2)\phi_3(x,y^3) \phi_4(x,y^4),
\end{equation}
where $x=(x^1,x^2,x^3,x^4)$, $y=(y^1,y^2,y^3,y^4)$, and if
 $0<\kappa(x^1,x^2)<1$, then
\begin{align}
&\displaystyle \phi_1(x,u):=\frac{p(x^1,u)-p(x^1,
u)\wedge p(x^2,u)}{1-\kappa(x^1,x^2)}, \quad
\phi_2(x,u):=\frac{p(x^2,u)-p(x^1,u)\wedge p(x^2,u)}{1-\kappa(x^1,x^2)},\label{phi_12}
 \\\nonumber\\
&\displaystyle \phi_3(x,u):=1(x^4=1)\frac{p(x^1,u)\wedge
 p(x^2,u)}{\kappa(x^1,x^2)}+1(x^4=0)p(x^3,u),\label{phi_3}
 \\\nonumber\\ 
&\displaystyle \phi_4(x,u):=1(x^4=1)\left(\delta_1(u)(1-\kappa(x^1,x^2))+ 
\delta_0(u)\kappa(x^1,x^2)\right) +1(x^4=0)\delta_0(u)\label{phi_4}, 
\end{align}
where $\delta_i(u)$ is the Kronecker symbol, $\delta_i(u) = 1(u=i)$, or, in other words, the delta measure concentrated at state $i$. The case $x^4=0$ signifies coupling which has already been realised at the previous step, and $u=0$ means successful coupling at the transition.  
Note that $\phi_1$ and $\phi_2$ do not depend on the variable $x^3$; we will highlight it by the notation $\phi_i((x^1,x^2,*,x^4),u)$ ($i=1,2$) where $*$ stands for any possible value of $x^3$. Also even if it is written $\phi_3((x^1,x^2,x^3,1),u)$, this value does not depend on $x^3$ either.

In the degenerate cases, if $\kappa(x^1,x^2)=0$ (coupling at the transition is impossible), then instead of (\ref{phi_3}) we set,  for example, 
\begin{align}\label{after0}
\phi_3(x,u):=1(x^4=1)p(x^3,u) + 1(x^4=0)p(x^3,u) = p(x^3,u),
\end{align}
and if $\kappa(x^1,x^2)=1$, then instead of (\ref{phi_12}) we may set 
\begin{align}\label{after}
\phi_1(x,u)=\phi_2(x,u):= p(x^1,u). 
\end{align}
The formula (\ref{phi_4}) which defines \(\phi_4(x,u)\) can be accepted in all cases. \\
Now let us define the process $(\widetilde X^1_n, \widetilde X^2_n)$,  $n\ge 0$ by the formulae
\begin{align}\label{xin}
\widetilde X^1_n:=\eta^1_n 1(\zeta_n=1)+\xi_n 1(\zeta_n=0), \quad 
\widetilde X^2_n:=\eta^2_n 1(\zeta_n=1)+\xi_n 1(\zeta_n=0).
\end{align}

\medskip

Looking at the construction, it may seem that the transition  densities for the components $\tilde X^1_{n+1}$ and $\tilde X^2_{n+1}$, respectively, given $\tilde X^1_{n}$ and $\tilde X^2_{n}$ may depend on both $\tilde X^1_{n}$ and $\tilde X^2_{n}$. This is not so, functionally the first one depends only on $\tilde X^1_{n}$, and the second one, respectively, on $\tilde X^2_{n}$. 
In fact, due to the construction above, we have, in particular, the following densities of the conditional distributions of $(\tilde X^1_{n+1}, \tilde X^2_{n+1})$ given $(\tilde X^1_{n}, \tilde X^2_{n})$ and $(\tilde X^1_{n} \not =\tilde X^2_{n})$ (in which case by definition $(\tilde X^1_{n},\tilde X^2_{n}) = (\eta^1_{n}, \eta^2_{n})$; the sign $*$ stands for any possible value in the range):
\begin{align*}
\frac{\mathbb P(\tilde X^1_{n+1} \in dx^1 |\tilde X^1_{n}, \tilde X^2_{n}, \tilde X^1_{n} \not =\tilde X^2_{n})}{\Lambda_{\tilde X^1_{n}, \tilde X^2_{n}}(dx^1)}
= \frac{\mathbb P(\tilde X^1_{n+1} \in dx^1 |\tilde X^1_{n}=\eta^1_{n}, \tilde X^2_{n}=\eta^2_{n}, \eta^1_{n} \not =\eta^2_{n})}
{\Lambda_{\eta^1_{n}, \eta^2_{n}}(dx^1)}
 \\\\
= (1-\kappa(\eta^1_n,\eta^2_n)) \phi_1((\eta^1_{n},\eta^2_{n},*,1),x^1)
+ \kappa(\eta^1_n,\eta_n^2) \phi_3((\eta^1_n,\eta^2_n,*,1),x^1)
 \\\\
=p(\eta^1_n,x^1)-p(\eta^1_n, x^1)\wedge p(\eta^2_n,x^1) + p(\eta^1_n, x^1)\wedge p(\eta^2_n,x^1) =p(\eta^1_n,x^1) = p(\tilde X^1_n,x^1), 
\end{align*}
due to (\ref{phi_12}), and similarly, given $(\tilde X^1_{n}, \tilde X^2_{n})$ and $(\tilde X^1_{n} \not =\tilde X^2_{n})$, 
\begin{align*}
\frac{\mathbb P(\tilde X^2_{n+1} \in dx^2 |\tilde X^1_{n}, \tilde X^2_{n},\tilde X^1_{n} \not =\tilde X^2_{n})}{\Lambda_{\tilde X^1_{n}, \tilde X^2_{n}}(dx^2)} =  p(\eta^2_n,x^2) =  p(\tilde X^2_n,x^2).
\end{align*}
Also, given 
$(\tilde X^1_{n}, \tilde X^2_{n})$ and $(\tilde X^1_{n} =\tilde X^2_{n})$ we can check that for any $z$ (which stands here both for $x^1$ and $x^2$) according to (\ref{phi_3}) we have, 
\begin{align*}
\frac{\mathbb P(\tilde X^1_{n+1} \in dz |\tilde X^1_{n}, \tilde X^2_{n},\tilde X^1_{n} =\tilde X^2_{n})}{\Lambda_{\tilde X^1_{n}, \tilde X^2_{n}}(dz)} 
= \frac{\mathbb P(\tilde X^2_{n+1} \in dz |\tilde X^1_{n}, \tilde X^2_{n},\tilde X^1_{n} =\tilde X^2_{n})}{\Lambda_{\tilde X^1_{n}, \tilde X^2_{n}}(dz)} =  p(\tilde X^1,z) =  p(\tilde X^2,z).
\end{align*}
Indeed, 
\begin{align*}
\frac{\mathbb P(\tilde X^1_{n+1} \in dx^1 |\tilde X^1_{n}, \tilde X^2_{n}, \tilde X^1_{n} =\tilde X^2_{n})}{\Lambda_{\tilde X^1_{n}, \tilde X^2_{n}}(dx^1)}
= \frac{\mathbb P(\tilde X^1_{n+1} \in dx^1 |\tilde X^1_{n}= \tilde X^2_{n} = \xi_{n})}
{\Lambda_{\eta^1_{n}, \eta^2_{n}}(dx^1)}
 \\\\
= \phi_3((*,*,\xi_n,0),x^1) 
=p(\xi_n,x^1) = p(\tilde X^1_n,x^1). 
\end{align*}
Therefore, we have
\begin{align*}
\frac{\mathbb P(\tilde X^1_{n+1} \in dx^1 |\tilde X^1_{n}, \tilde X^2_{n})}{\Lambda_{\tilde X^1_{n}, \tilde X^2_{n}}(dx^1)} 
= \frac{\mathbb P(\tilde X^1_{n+1} \in dx^1 |\tilde X^1_{n}, \tilde X^2_{n})}{\Lambda_{\tilde X^1_{n}, \tilde X^2_{n}}(dx^1)}(1(\tilde X^1_{n} =\tilde X^2_{n}) + 1(\tilde X^1_{n} \not =\tilde X^2_{n})) 
 \\
= 1(\tilde X^1_{n} =\tilde X^2_{n})\frac{\mathbb P(\tilde X^1_{n+1} \in dx^1 |\tilde X^1_{n}, \tilde X^2_{n}, \tilde X^1_{n} =\tilde X^2_{n})}{\Lambda_{\tilde X^1_{n}, \tilde X^2_{n}}(dx^1)} 
 \\
+ 1(\tilde X^1_{n} \not =\tilde X^2_{n}) \frac{\mathbb P(\tilde X^1_{n+1} \in dx^1 |\tilde X^1_{n}, \tilde X^2_{n}, \tilde X^1_{n} \neq \tilde X^2_{n})}{\Lambda_{\tilde X^1_{n}, \tilde X^2_{n}}(dx^1)}
 \\
= p(\tilde X^1_n,x^1) \left(1(\tilde X^1_{n} =\tilde X^2_{n}) + 1(\tilde X^1_{n} \not =\tilde X^2_{n})\right)  = p(\tilde X^1_n,x^1),
\end{align*}
in all cases.

Due to all of these, each of the components $\tilde X^1_{n}$ and $\tilde X^2_{n}$ are Markov processes with the same generator as $X^1_{n}$ and $X^2_{n}$. (NB. The little calculus above is, of course, not the proof of the Markov property, which just follows from the construction itself;  rather these formulae show how to understand the transition probability kernels of the chosen coupling algorithm.) Moreover, the following lemma holds true. 
\begin{Lemma}\label{lemma:2}
Let the random variables $\widetilde X^1_n$ and $\widetilde X^2_n$ be defined for $n\in Z_+$ by the formulae (\ref{xin}). 
Then 
\begin{equation}\label{l71}
\widetilde X^1_n\stackrel{d}{=}X^1_n, \;\;\widetilde
 X^2_n\stackrel{d}{=}X^2_n, \quad \mbox{for all $n\ge 0$,}
\end{equation}
which implies that the process $\widetilde X^1$ is equivalent to $X^1$, and the process $\widetilde X^2$ is equivalent to $X^2$ in distribution in the space of trajectories; in particular, each of them is a Markov process with the same generator as $X^1$.
Moreover, the couple $\tilde X_n:=\left(\widetilde X^1_n, \widetilde X^2_n\right)$, $n\ge 0$, is also  a  homogeneous Markov process, and 
\begin{equation*}\label{mpequi}
\left(\widetilde X^1_n\right)_{n\ge 0}\stackrel{d}{=}\left(X^1_n\right)_{n\ge 0}, \quad \& \quad 
\left(\widetilde X^2_n\right)_{n\ge 0}\stackrel{d}{=}\left(X^2_n\right)_{n\ge 0}.
\end{equation*}
Moreover, 
\begin{equation}\label{l72}
\widetilde X^1_n=\widetilde X^2_n, \quad \forall \; n\ge
n_0(\omega):=\inf\{k\ge0: \zeta_k=0\}, 
\end{equation}
and
\begin{equation}\label{estimate}
\P_{x^1,\mu}(\widetilde X^1_n\neq \widetilde X^2_n) 
\le \E_{x^1,\mu}\prod_{i=0}^{n-1} (1-\kappa(\tilde X^1_i,\tilde X^2_i))
\le \E_{x^1,\mu}\prod_{i=0}^{n-1} (1-\kappa(\eta^1_i,\eta^2_i)).
\end{equation}
\end{Lemma}

Very informally, the processes $\eta^1_n$ and $\eta^2_n$ represent $X^1_n$ and $X^2_n$, correspondingly,
under the condition that the coupling was not successful until time $n$. The process $\xi_n$ evolves independently of them, representing both  $X^1_n$ and $X^2_n$ simultaneously after the coupling occurs. 
The process $\zeta_n$ represents the moment of coupling: the event  $\zeta_n=0$ is equivalent to the event that coupling occurs no later than at time $n$, while  $\zeta_n=1$ is the complementary event. 
We also note that it is possible that with a positive probability at some moment (stopping time) $X^1_n \not = X^2_n$, but $\kappa(X^1_n,X^2_n) = 1$; in this case coupling occurs on the next step with probability one, the further representation of both  $X^1_n$ and $X^2_n$ is taken over by the process $\xi_n$, as said earlier, and the couple $(\eta^1_n,\eta^2_n)$ further evolves according to (\ref{after}), while the product $\prod_{i=0}^{n}
 (1-\kappa(\eta^1_i,\eta^2_i))$ takes value zero. The couple $(\eta^1_n,\eta^2_n)$ remains a homogeneous Markov process; an alternative construction could have been a transition of $(\eta^1_n,\eta^2_n)$ from the state where $\kappa=1$ to the infinite state $\partial_\infty$ which is absorption.

The non-homogeneous version of this lemma automatically includes the homogeneous scenario and will be proved in the section \ref{sec:nonh}.

\medskip

\subsection{Operators $V$ and $\hat V$ and their spectral radii}\label{sec:Vr}

Now using the coefficient $\kappa(x)$, let us introduce one more key notion of the operator $V$ acting on a (bounded, Borel measurable) function $h$ on the space $S^2 := S\times S$ as follows: for $x=(x^1, x^2)\in S^2$,
\begin{equation}\label{V}
Vh(x) := (1-\kappa(x^1,x^2)) \mathbb E_{x^1,x^2}h(\tilde X_1) \equiv \exp(\psi(x))\mathbb E_{x^1,x^2}h(\tilde X_1),  
\end{equation}
where in the last expression $\psi(x):= \ln (1-\kappa(x^1,x^2))$ (assume $\ln 0 = -\infty$); recall that 
$\tilde X_n = (\tilde X_n^1, \tilde X^2_n)$. Note that on the diagonal $x=(x^1,x^2) : x^1=x^2$ we have 
$$
Vh(x) = (1-\kappa(x^1,x^1))\mathbb E_{x^1,x^2}h(\tilde X_1) = 0, 
$$
since $\kappa(x^1,x^1) = 1$ for any $x^1$. This corresponds well to the idea of coupling: where the two processes $\tilde X^1$ and $\tilde X^2$ become equal, coupling occurs (or it has occurred earlier). Hence, it makes sense to either consider the functions $h$ on $S^2$ vanishing on the diagonal $\text{diag}(S^2)= (x =  (x^1,x^1)\in S^2)$, or, equivalently, to reduce the operator itself on functions defined on 
$$
\hat S^2 := S^2 \setminus \text{diag}(S^2), 
$$
that is, to define for $x=(x^1,x^2)\in \hat S^2$ and for functions $\hat h: \hat S^2 \to \mathbb R$, 
\begin{equation}\label{Vhat}
\hat V\hat h(x):= (1-\kappa(x^1,x^2))\mathbb E_{x^1,x^2}\hat h(\tilde X_1)1(x^1\not =x^2).
\end{equation}

The estimate (\ref{estimate}) can be rewritten via the operator $V$, or, equivalently,  via $\hat V$ as follows: 
\begin{align}\label{l2}
\!\!\P_{x^1,\mu}(\widetilde X^1_n\neq \widetilde X^2_n)\!\le\! 
\int \E_{x^1,x^2} V^n {\bf 1}(x^1,x^2)1(x^1\!\not =\!x^2)\mu(dx^2) 
 \nonumber\\  \\\nonumber
=\int \E_{x^1, x^2}1(\tilde X_0^1\!\not = \!\tilde X_0^2) \hat V^n {\bf 1}(x^1,x^2)1(x^1\!\not =\!x^2)\mu(dx^2). 
\end{align}
Note that by definition (\ref{V}), for the non-negative operator $V$ (which transforms any non-negative function into a non-negative one) its norm $\|V\| = \|V\|_{B,B}:=\sup\limits_{|h|_B\le 1} |Vh|_B $ equals $\sup\limits_{x \in S^2} V{\bf 1}(x)$, where $|h|_B := \max\limits_{x\in S^2} |h(x)|$ (the sup-norm), and ${\bf 1}=(1(x)=1, \, x\in S^2)$. Thus, 
$$
\|V\| = \sup\limits_{x\in S^2} V1(x) = \sup_{x\in S^2} (1-\kappa(x)) = 1-\kappa.
$$ 

\noindent
Now the well-known inequality (see, for example, \cite[\S 8]{KLS}) reads, 
\begin{equation*}\label{rkappa}
r(V) \le \|V\| \; = (1-\kappa). 
\end{equation*}
The same holds true for the operator $\hat V$ (here the function ${\bf 1}(x)\equiv 1$ is defined on $\hat S^2$): 
\begin{equation*}\label{rkappa}
r(\hat V) \le \|\hat V\| = \sup\limits_{x\in \hat S^2} \hat V{\bf 1}(x) = \sup_{x\in \hat S^2} (1-\kappa(x)) = 1-\kappa. 
\end{equation*}

Further, if $V$ (or  $\hat V$) were compact 
and irreducible 
(see, e.g., \cite{KLS}) then from the generalisation of the Perron--Frobenius Theorem (see, for example, \cite[\S 9, Theorem 9.2]{KLS}) it would follow (see, e.g., \cite[(7.4.10)]{FW}), 
$$
\lim_{n \to\infty} \frac1n \, \ln V^n {\bf 1}(x)  = \ln r(V) = \ln r(\hat V) = \lim_{n \to\infty} \frac1n \, \ln \hat V^n \hat{\bf 1}(x) 
$$
(where the first unit function $\bf 1$ is defined on $S^2$, while the second one $\hat{\bf 1}(x)$ on $\hat S^2$). \\
Similarly to the above, an operator for $m$ steps may be defined on functions $h: S^2 \to \mathbb R$:
$$
V^{(m)} h(x) := (1-\kappa^{(m)}(x^1,x^2)) \mathbb E_{x^1,x^2}h(\tilde X_m) \equiv \exp(\psi_m(x))\mathbb E_{x^1,x^2}h(\tilde X_m),  
$$
and likewise 
on $h: \hat S^2 \to \mathbb R$
\begin{align*}
\hat V^{(m)} h(x) := (1-\kappa^{(m)}(x^1,x^2)) \mathbb E_{x^1,x^2}h(\tilde X_m)1(\tilde X_m^1 \neq \tilde X_m^2) 
 \\\\
\equiv \exp(\psi_m(x))\mathbb E_{x^1,x^2}h(\tilde X_m)1(\tilde X_m^1 \neq \tilde X_m^2).  
\end{align*}
Note that the operator $V^n$ ($n\ge 1$) maps the values of any function at the diagonal of $S^2$ to zero: we have for $x=(x^1,x^1)$ 
$$
V h(x) = (1-\kappa(x^1,x^1))\mathbb E_x h(\tilde X_1) = 0, 
$$
since $1-\kappa(x^1,x^1) =0$. Also, the operator  $V$ is positive, that is, $Vh(x)\ge 0$ for any function $h\ge 0$. Therefore (here ${\bf 1}(x)$ is the function identically equal to one on $S^2$, and $\hat {\bf 1}(x)$  is the function identically equal to one on $\hat S^2$),  
$$
\|V^n\| = \sup_{x\in S^2} V^n  {\bf 1}(x) = \sup_{x\in S^2}(1-\kappa^{(n)}(x)) = \sup_{\hat x\in \hat S^2}(1-\hat \kappa^{(n)}(\hat x)) = \sup_{\hat x\in \hat S^2} \hat V^n\hat  {\bf 1}(\hat x) = \|\hat V^n\|. 
$$
Hence we have the equality 
$r(V) = r(\hat V)$. 
It will be easier to argue with the operator $V$ in the sequel; so, we will mainly continue with  this operator. However, from the computational point of view -- that is, to compute the spectral radius -- the operator $\hat V$ is preferred because of some reduction in dimension. Recall  that 
$\hat V$ is the reduction of $V$ on $\hat S^2$, and that 
$$
\hat V^n {\bf 1}(x) = V^n {\bf 1}(x), \quad x\in \hat S^2,
$$
and 
$$
\ln r(V) = \ln r(\hat V) \le \|\hat V\| = \|V\|.
$$

Note that even  without the assumption of compactness on $V$ we still have the inequalities, 
\[
0\le \lim_{n \to\infty} \frac1n \, \ln V^n {\bf 1}(x) \le \lim_{n \to\infty} \frac1n \, \ln \|V^n\| = \ln r(V) \le \|V\|.
\]
So, from the Gelfand formula, 
\begin{equation}\label{l11}
\limsup_n (V^n {\bf 1} (x))^{1/n} \le \lim_n \|V^n\|^{1/n} = r(V). 
\end{equation}
The assertions (\ref{l2}) and (\ref{l11}) together lead to the following result. 
\begin{Theorem}\label{lastthm2}
In all cases, for any $x^1 \in S$,
\begin{equation}\label{newrate12}
\limsup\limits_{n\to\infty} \frac1n \ln \| \mathbb P_{x^1}(n,\cdot) - \mu(\cdot)\|_{TV} 
\!\le\! \limsup_{n\to\infty} \frac1n  \ln\! \int \!\!2V^n {\bf 1}(x^1, x^2)\mu(dx^2)
\!\le\! \ln r(V).
\end{equation}
\end{Theorem}
\noindent
{\em Proof.} We have, due to (\ref{l2}),

\begin{align*}
\limsup\limits_{n\to\infty} \frac1n \ln \| P_{x^1,\mu}(n,\cdot) - \mu(\cdot)\|_{TV} 
\le \limsup\limits_{n\to\infty} \frac1n \ln (2\P_{x^1,\mu}(\widetilde X^1_n\neq \widetilde X^2_n))
 \\\\
= \limsup\limits_{n\to\infty} \frac1n \ln (2 \int \P_{x^1,x^2}(\widetilde X^1_n\neq \widetilde X^2_n)\mu(dx^2))
  \\\\
\le  \limsup_n \frac1n \, \ln \int 2V^n {\bf 1}(x^1, x^2) 1(x^2\not = x^1)\mu(dx^2)
\le \ln r(V).  
\end{align*}
\begin{Corollary}\label{corr1}
Under the assumption
\begin{equation}\label{r12}
r(V)<1, 
\end{equation}
the rate of convergence in
\[
\|\mu_n - \mu\|_{TV} \to 0, \quad n\to\infty
\]
is exponential: 
for any $\epsilon>0$ and $n$ large enough ($n\ge N(x^1)$), 
\begin{equation}\label{newrate2}
\|P_x(n,\cdot) - \mu(\cdot)\|_{TV} \le (r(V)+ \epsilon)^n.  
\end{equation}
\end{Corollary}

\noindent
Here one may expect the multiplier $2$ in the right hand side, but it is dropped due to $+\epsilon$. 
To put it a little differently, if $r(V)<\|V\| = 1-\kappa$ and $0 < \epsilon < 1-\kappa - r(V)$, then the bound (\ref{newrate2}) is strictly better than (\ref{exp_bd32}) for $n$ large enough. Everywhere $r(\hat V)$ can be used equivalently.

\begin{Remark}
Let us emphasize that on the one hand the bound (\ref{newrate12}) in the Theorem \ref{lastthm2} is asymptotic, for large $n$, unlike the strict bounds in the classical Ergodic Theorem and in the Diaconis--Stroock bound for reversible MC (see  (\ref{dse3}) below). On the other hand, the estimates  
$$
\frac1n \ln \| P_{x^1}(n,\cdot) - \mu(\cdot)\|_{TV} 
\le \frac1n \ln (2\P_{x^1,\mu}(\widetilde X^1_n\neq \widetilde X^2_n))
\le  \frac1n \, \ln \int 2V^n{\bf 1}(x^1, x^2) 1(x^1\neq  x^2)\mu(dx^2)
$$ 
are valid for each $n\ge 1$ and $x^1\in S$.
\end{Remark}

\begin{Remark}
The condition (\ref{r12}) offers one possible (partial) answer to the question whether there is any intermediate situation in ``between'' Markov--Dobrushin's and Doeblin--Doob's with a bound like Doeblin--Doob's (see \cite{Doob53}): in the last case the following bound holds true
\begin{equation}\label{DD}
 \sup_{x}\sup_{A\subset {\cal S}} |\P_x(n,A) - \mu(A)| \le C\exp(-cn), \quad n\ge 0, 
\end{equation}
with some $C,c>0$, under the ``DD-condition'' which assumes that there exist a finite (sigma-additive) measure $\nu\ge 0$ and $\epsilon>0$, $s>0$ such that $\nu(A)\le \epsilon$ implies 
\[
\sup_x \P_x(s, A) \le 1 - \epsilon. 
\]
Some issue with the bound (\ref{DD}) is that the constants $C,c$ are neither determined by the measure $\nu$ and the constant $\epsilon$, nor any bounds for these constants exist in terms of $\nu$ and $\epsilon$. 

Beside the examples in the last section, note that if the MD condition $\kappa>0$ fails, it means $\kappa=0$, which just signifies that for at least one couple of states $x$ and $x'$ the kernels $Q_x(dy)$ and $Q_{x'}(dy)$ are singular, but it does not necessarily mean $r(V)=1$ since the process still may well be irreducible. So, indeed, the inequality $r(V)<1$ provides a condition, which is more relaxed than MD and yet the one which allows an effective bound for the rate of convergence. 

\end{Remark}

\begin{Theorem}\label{lastthmm2m}
In all cases, 
\begin{align*}\label{}
\limsup\limits_{n\to\infty} \frac1n \ln \| P_x(n,\cdot) - \mu(\cdot)\|_{TV} 
 \\\\
\le \limsup_n \frac1n \, \ln \int 2(V^{(m)})^{[n/m]} {\bf 1}(x) 1(x^1\not = x^2)\mu(dx^2)
\le \ln r(V^{(m)})^{1/m}.  
\end{align*}
\end{Theorem}
\begin{Remark}
Also, 
$$
r(V^{(m)})^{1/m} = r(V), 
$$
because $V^{(m)} = V^m$, and, hence, 
$$
\lim_{n\to\infty} \|(V^{(m)})^n\|^{1/n} = \lim_{n\to\infty} \|V^{nm}\|^{1/n}
= r(V)^{m}. 
$$
\end{Remark}

\section{Non-homogeneous case}\label{sec:nonh}
Here certain non-homogeneous cases are discussed where analogues of the approach developed in the previous section can be applied. In general this does not seem possible. However, what is possible is to assume that in all  non-homogeneous transition kernels $\P_{t,x}(t+1,x')$ there is a nontrivial homogeneous core sub-kernel, to which it might be possible to apply the  approach on the base of the previous section; this will be realised under the assumptions (\ref{ppt3}) -- (\ref{ppt5}) below. Also, a periodic dependence may hold true for the kernels, see (\ref{T}) below; of course, this case may also be treated as a $T$-dependent homogeneous Markov chain.
Note that the notion of the ``joint spectral radius''  in the present situation looks  regretfully  useless (see \cite{Koz}).

\subsection{Auxiliaries}

Let us now consider a non-homogeneous Markov process (MP) in discrete time $(X_n, \, n\ge 0)$ on a general state space $S$ with a topology and with a Borel sigma-algebra.

In the well-known inequality (\ref{exp_bd3}) of the Proposition \ref{thm_erg2}  stated below for the reader's convenience (without proof: see \cite[Section 4]{Kolm}), which extends the bound (\ref{finiteMD}), Markov--Dobrushin's coefficients depending on time are used, 
\begin{equation}\label{MDt}
 \kappa_t := \inf_{x,x'} \int \left(\frac{\P_{t,x'}(t+1,dy)}{\P_{t,x}(t+1,dy)}\wedge 1 \right)\P_{t,x}(t+1,dy). 
\end{equation}
Similar coefficient and condition can be introduced for any number $m\ge 1$ of steps:
$$
 \kappa_t^{(m)} := \inf_{x,x'} \int \left(\frac{\P_{t,x'}(t+m,dy)}{\P_{t,x}(t+m,dy)}\wedge 1 \right)\P_{t,x}(t+m,dy). 
$$
Also denote
$$
 \kappa_t (x,x') = \int \left(\frac{\P_{t,x'}(t+1,dy)}{\P_{t,x}(t+1,dy)}\wedge 1 \right)\P_{t,x}(t+1,dy),  
$$
and 
$$
 \kappa_t^{(m)}(x,x') :=  \int \left(\frac{\P_{t,x'}(t+m,dy)}{\P_{t,x}(t+m,dy)}\wedge 1 \right)\P_{t,x}(t+m,dy). 
$$
Note that here, as for the homogeneous case, $\displaystyle \P_{t,x'}(t+1,dy)/\P_{t,x}(t+1,dy)$ is understood in the sense of the density of the absolutely continuous component of the numerator with respect to the denominator measure. Note that the probability $\P_{t,x}(t+1,A)$ for any Borel measurable $A$ is, in turn, Borel measurable with respect to $x$, which is a standard requirement in Markov processes \cite{Dynkin}. Due to linearity such measurability with respect to the pair $x,x'$ will be also valid for the measure $\Lambda_{t,x,x'}$ defined below.

~

For any two or three fixed states $x^1, x^2, x^3$ and for any $t\ge 0$ denote 
\begin{equation*}
\Lambda_{t,x^1,x^2}(dz) := (\P_{t,x^1}(t+1,dz) + \P_{t,x^2}(t+1,dz))/2,
\end{equation*} 
\begin{equation*}
\Lambda_{t,x^1,x^2, x^3}(dz) := (\P_{t,x^1}(t+1,dz) + \P_{t,x^2}(t+1,dz) + \P_{t,x^3}(t+1,dz))/3.
\end{equation*} 
Likewise, for any $m\ge 1$, let 
\begin{equation*}
\Lambda^{(m)}_{t,x^1,x^2}(dz) := (\P_{t,x^1}(t+m,dz) + \P_{t,x^2}(t+m,dz))/2,
\end{equation*} 
\begin{equation*}
\Lambda^{(m)}_{t,x^1,x^2, x^3}(dz) := (\P_{t,x^1}(t+m,dz) + \P_{t,x^2}(t+m,dz) + \P_{t,x^3}(t+m,dz))/3.
\end{equation*} 
\noindent
Note that $\Lambda_{t,x^1,x^2}(dz) = \Lambda_{t,x^1,x^2}(dz)$, and similarly $\Lambda_{t,x^1,x^2,x^3}$ as well as $\Lambda^{(m)}_{t,x^1,x^2,x^3}$ do not depend on the permutation of the variables $(x^1,x^2,x^3)$.
\begin{Lemma}\label{newMDt}
The following representation for the condition (\ref{MDt}) holds true, 
\begin{equation}\label{betterMD}
 \kappa_t = \inf_{x,x'} \int \left(\frac{\P_{t,x'}(t+1,dy)}{\Lambda_{t,x,x'}(dy)}\wedge \frac{\P_{t,x}(t+1,dy)}{\Lambda_{t,x,x'}(dy)} \right)\Lambda_{t,x,x'} (dy). 
\end{equation}
In particular -- since $\Lambda_{t,x,x'}(dz) = \Lambda_{t,x',x}(dz)$ -- 
for any $x,x'\in S$,
\[
\kappa_t(x,x') = \kappa_t(x',x).
\]
\end{Lemma}

\noindent
{\em Proof. } 
Let $\displaystyle f_{x,x'}(y) = \frac{\P_{t,x}(t+1,dy)}{\Lambda_{t,x,x'} (dy)}(y)$. Then, 
\begin{eqnarray*}
&\displaystyle  \kappa_t = \inf_{x,x'} \int \left(\frac{\P_{t,x'}(t+1,dy)}{\P_{t,x}(t+1,dy)}\wedge \frac{\P_{t,x}(t+1,dy)}{\P_{t,x}(t+1,dy)} \right)\P_x(dy) 
 \\\\
&\displaystyle  = \inf_{x,x'} \int \left(\frac{\P_{x'}(dy)}{f_{x,x'}(y)\Lambda_{x,x'}(dy)}\wedge \frac{\P_{x}(dy)}{f_{x,x'}(y)\Lambda_{x,x'}(dy)} \right)f_{x,x'}(y)\Lambda_{x,x'} (dy) 
 \\\\
&\displaystyle = \inf_{x,x'} \int \left(\frac{\P_{x'}(dy)}{\Lambda_{x,x'}(dy)}\wedge \frac{\P_{x}(dy)}{\Lambda_{x,x'}(dy)} \right)\Lambda_{x,x'} (dy), 
\end{eqnarray*}
as required. The Lemma \ref{newMDt} is proved. 

The same is valid for $\kappa^{(m)}_t$ for any $m$ where
$$
 \kappa^{(m)}_t = \inf_{x,x'} \int \left(\frac{\P_{t,x'}(t+m,dy)}{\Lambda^{(m)}_{t,x,x'}(dy)}\wedge \frac{\P_{t,x}(t+m,dy)}{\Lambda^{(m)}_{t,x,x'}(dy)} \right)\Lambda^{(m)}_{t,x,x'} (dy).
$$

\begin{Remark}
Actually, the right hand side in (\ref{betterMD}),  does not depend on any particular reference measure $\Lambda_{t,x,x'}$ (even if it is not symmetric with respect to $x,x'$). Indeed, it follows straightforwardly from the fact that if $d\Lambda_{t,x,x'} <\!\!< d\tilde\Lambda_{t,x,x'}$ and $d\Lambda_{t,x,x'} = \phi_{t,x,x'} d\tilde\Lambda_{t,x,x'}$, then we get, 
\begin{eqnarray*}
&\displaystyle  \int \left(\frac{\P_{t,x'}(t+1,dy)}{\Lambda_{t,x,x'}(dy)}\wedge \frac{\P_{t,x}(t+1,dy)}{\Lambda_{t,x,x'}(dy)} \right)\Lambda_{t,x,x'} (dy) 
 \\\\
&\displaystyle  = \int \left(\frac{\P_{t,x'}(t+1,dy)}{\tilde\Lambda_{t,x,x'}(dy)}\wedge \frac{\P_{t,x}(t+1,dy)}{\tilde\Lambda_{t,x,x'}(t+1,dy)} \right)1(\phi_{t,x,x'}(y)>0)\tilde\Lambda_{t,x,x'} (dy).
\end{eqnarray*}
However, $\P_{t,x'}(t+1,dy)<\!\!< \Lambda_{t,x,x'}(t+1,dy) = \phi_{t,x,x'}(y) \tilde \Lambda_{t,x,x'}(t+1,dy)$, so for any measurable $A$ we have $\int_A \P_{t,x'}(t+1,dy) 1(\phi_{t,x,x'}(y)=0) = 0$ and the same for $\P_{t,x}(t+1,dy)$, which means that, in fact, 
\begin{align*}
\int \left(\frac{\P_{t,x'}(t+1,dy)}{\tilde\Lambda_{t,x,x'}(dy)}\wedge \frac{\P_{t,x}(t+1,dy)}{\tilde\Lambda_{t,x,x'}(t+1,dy)} \right)1(\phi_{t,x,x'}(y)>0)\tilde\Lambda_{t,x,x'} (dy) 
 \\\\
= \int \left(\frac{\P_{t,x'}(t+1,dy)}{\tilde\Lambda_{t,x,x'}(dy)}\wedge \frac{\P_{t,x}(t+1,dy)}{\tilde\Lambda_{t,x,x'}(dy)} \right)\tilde\Lambda_{t,x,x'} (dy).
\end{align*}
Respectively, if there are two reference measures $\Lambda_{t,x,x'}$ and, say, $\Lambda'_{t,x,x'}$, then we may take $\tilde\Lambda_{t,x,x'} = \Lambda_{t,x,x'} + \Lambda'_{t,x,x'}$, and the coefficients computed by using each of the two -- $\Lambda_{t,x,x'}$ and $\Lambda'_{t,x,x'}$ -- will be represented via $\tilde\Lambda_{x,x'}$ in the same way. 
\end{Remark}

~

Here is the key notion in the following presentation: denote 
\[
\kappa_t(x,x') : = \int \left(\frac{\P_{t,x'}(t+1,dy)}{\P_{t,x}(t+1,dy)}\wedge 1 \right)\P_{t,x}(t+1,dy).
\]
Also, let 
\[
\kappa_t^{(m)}(x,x') : = \int \left(\frac{\P_{t,x'}(t+m,dy)}{\P_{t,x}(t+m,dy)}\wedge 1 \right)\P_{t,x}(t+m,dy).
\]
Clearly, for any $x,x'\in S$, 
\begin{equation*}
\kappa_t(x,x') \ge \kappa_t. 
\end{equation*}

\begin{Lemma}\label{jj}
For any $x,x'\in S$, and for any $m\ge 1$
\[
\kappa^{(m)}_t(x,x') = \kappa^{(m)}_t(x',x).
\]
\end{Lemma}
\noindent
{\em Proof.} 
We have, 
 
\begin{eqnarray*}
\displaystyle \kappa^{(m)}_t(x',x) = \int \left(\frac{\P_{t,x'}(t+m,dy)}{\P_{t,x}(t+m,dy)}\wedge 1 \right)\P_{t,x}(t+m,dy) 
 \nonumber \\ \nonumber \\
= \int \left(\frac{\P_{t,x'}(t+m,dy)}{\P_{t,x}(t+m,dy)}\wedge 1 \right)\frac{\P_{t,x}(t+m,dy)}{\Lambda^{(m)}_{t,x,x'}(dy)} \Lambda^{(m)}_{t,x,x'}(dy)
 \nonumber \\ \nonumber \\
\displaystyle = \int \left(\frac{\P_{t,x'}(t+m,dy)}{\Lambda^{(m)}_{t,x,x'}(dy)} \wedge \frac{\P_{t,x}(t+m,dy)}{\Lambda^{(m)}_{t,x,x'}(dy)} \right) \Lambda^{(m)}_{t,x,x'}(dy). 
\end{eqnarray*}
The latter expression is  symmetric with respect to $x$ and $x'$, which proves the Lemma \ref{jj}. 
The following proposition follows from the calculus similar to that for the homogeneous case.
\begin{Proposition}\label{thm_erg2}
Let the values of $\kappa_t$ be defined as in  (\ref{betterMD}). Then 
the uniform bound is satisfied for every $n$, 
\begin{equation}\label{exp_bd3}
 \sup_{x, x'}\sup_{A\subset S} |\P_{0,x}(n,A) - \P_{0,x'}(n,A)| \le \prod_{t=0}^{n-1} (1-\kappa_t)^{}.
\end{equation}
Also, for any $m\ge 1$
$$
\sup_{x,x'}\sup_{A\subset S} |\mu^x_n(A) - \mu^{x'}_n(A)| \le   \prod_{t=0}^{[(n-1)/m]} (1-\kappa_{tm}^{(m)}), 
$$
and 
\begin{equation*}
\sup_{x,x'}\|\mu^x_n - \mu^{x'}_n\|_{TV} \le 2  \prod_{t=0}^{[(n-1)/m]} (1-\kappa_{tm}^{(m)}). 
\end{equation*}

\end{Proposition}
\subsection{Markov coupling (non-homogeneous)}
The Lemma \ref{odvuh} can be generalised to a sequence of random variables: the coupling construction for Markov chains based on \cite{Vaserstein} is presented in what follows. Consider two versions $(X^1_n), (X^2_n)$ of the same Markov process with two initial distributions $\mu_0^1$ and  $\mu_0^2$ respectively (this does not exclude the case of non-random initial states). Denote 
\begin{equation*}
\kappa(0) :=
\int \left(\frac{\mu_0^1(dy)}{\mu_0^2(dy)}\wedge 1 \right)\mu_0^2(dy).
\end{equation*}
It is clear that $0\le \kappa(0)\le1$, and similarly for $\kappa(u,v)$ for all $u,v$.  
We assume that $X^1_0$ and $X^2_0$ have different distributions, so $\kappa(0)<1$. Otherwise we obviously have
$X^1_n\stackrel{d}{=}X^2_n$ (equality in distribution) for all $n$, and the coupling can be made trivially, just by letting  $\widetilde X^1_n= \widetilde
X^2_n:=X^1_n$. 

The random variables $\left(\eta^1_0,\eta^2_0,\xi_0,\zeta_0\right)$ are chosen directly on the basis of the Lemma \ref{odvuh} as $\left(\eta^1,\eta^2,\xi,\zeta\right)$, according to the distributions in (\ref{laws}), exactly as in the homogeneous case. 
Further, by induction, assuming that the random variables $\left(\eta^1_n,\eta^2_n,\xi_n,\zeta_n\right)$ have been determined for some $n$, let us show how to construct them for $n+1$. For this aim, let us define the transition probability density $\phi_n$ with respect to the measure $\Lambda_{n,x^1, x^2}$ (more accurately, with respect to $\Lambda_{n,x^1, x^2} \times \Lambda_{n,x^1, x^2} \times \Lambda_{n,x^1, x^2}\times (\delta_0 + \delta_1)/2$) for this (vector-valued) process as follows,
\begin{equation}\label{process_eta2}
\phi_t(x,y):=\phi_{1,t}(x,y^1)\phi_{2,t}(x,y^2)\phi_{3,t}(x,y^3) \phi_{4,t}(x,y^4),
\end{equation}
where $x=(x^1,x^2,x^3,x^4)$, $y=(y^1,y^2,y^3,y^4)$, and if
 $0<\kappa_t(x^1,x^2)<1$, then
\begin{align}
&\displaystyle \phi_{t,1}(x,u):=\frac{p_t(x^1,u)-p_t(x^1,
u)\wedge p_t(x^2,u)}{1-\kappa_t(x^1,x^2)}, \quad
\phi_{t,2}(x,u):=\frac{p_t(x^2,u)-p_t(x^1,u)\wedge p_t(x^2,u)}{1-\kappa_t(x^1,x^2)},\label{phi_1211}
 \\\nonumber\\
&\displaystyle \phi_{t,3}(x,u):=1(x^4=1)\frac{p_t(x^1,u)\wedge
 p_t(x^2,u)}{\kappa_t(x^1,x^2)}+1(x^4=0)p_t(x^3,u),\label{phi_311}
 \\\nonumber\\ 
&\displaystyle \phi_{t,4}(x,u):=1(x^4=1)\left(\delta_1(u)(1-\kappa_t(x^1,x^2))+ 
\delta_0(u)\kappa_t(x^1,x^2)\right) +1(x^4=0)\delta_0(u)\label{phi_411}. 
\end{align}
The case $x^4=0$ signifies coupling which has already been realised at the previous step, and $u=0$ means successful coupling at the transition.  
Just as in the homogeneous case, $\phi_{t,1}$ and $\phi_{t,2}$ do not depend on the variable $x^3$; we will denote it by the notation $\phi_{t,i}((x^1,x^2,*,x^4),u)$ ($i=1,2$) where $*$ stands for any possible value of $x^3$. Also even if it is written $\phi_{t,3}((x^1,x^2,x^3,1),u)$, yet, this value does not depend on $x^3$ either.

In the degenerate cases, if $\kappa_t(x^1,x^2)=0$ (coupling at the transition is impossible), then instead of (\ref{phi_311}) we set,  e.g., 
$$
\phi_{t,3}(x,u):=1(x^4=1)p_t(x^3,u) + 1(x^4=0)p_t(x^3,u) = p_t(x^3,u),
$$ 
and if $\kappa_t(x^1,x^2)=1$, then instead of (\ref{phi_1211}) we may set 
$$
\phi_{t,1}(x,u)=\phi_{t,2}(x,u):= p_t(x^1,u). 
$$ 
Note that in the case of $\kappa(x^1,x^2)=1$ we shall not assume, in particular, that the ``next'' values $\eta^1_{n+1}$ and  $\eta^2_{n+1}$ are necessarily different, and in the case of existence of a non-empty set of such pairs $x^1 \neq x^2$ it is not convenient to consider the process $(\eta^1_n,\eta^2_n)$ on the truncated state space $\hat S^2$; however, after the moment of coupling it is of no importance for our estimate.  
The formula (\ref{phi_411}) which defines \(\phi_4(x,u)\) can be accepted in all cases.

\medskip

Similarly to the homogeneous case, we have
\begin{align*}
\frac{\mathbb P(\tilde X^1_{n+1} \in dx^1 |\tilde X^1_{n}, \tilde X^2_{n}, \tilde X^1_{n} \not =\tilde X^2_{n})}{\Lambda_{n,\tilde X^1_{n}, \tilde X^2_{n}}(dx^1)}
= \frac{\mathbb P(\tilde X^1_{n+1} \in dx^1 |\tilde X^1_{n}=\eta^1_{n}, \tilde X^2_{n}=\eta^2_{n}, \eta^1_{n} \not =\eta^2_{n})}
{\Lambda_{n,\eta^1_{n}, \eta^2_{n}}(dx^1)}
= p_n(\tilde X^1_n,x^1), 
\end{align*}
due to (\ref{phi_1211}), and similarly, given $(\tilde X^1_{n}, \tilde X^2_{n})$ and $(\tilde X^1_{n} \not =\tilde X^2_{n})$, 
\begin{align*}
\frac{\mathbb P(\tilde X^2_{n+1} \in dx^2 |\tilde X^1_{n}, \tilde X^2_{n},\tilde X^1_{n} \not =\tilde X^2_{n})}{\Lambda_{n,\tilde X^1_{n}, \tilde X^2_{n}}(dx^2)} =  p_n(\eta^2_n,x^2) =  p_n(\tilde X^2_n,x^2).
\end{align*}
Also, given 
$(\tilde X^1_{n}, \tilde X^2_{n})$ and $(\tilde X^1_{n} =\tilde X^2_{n})$ we can check that for any $z$ (which stands here both for $x^1$ and $x^2$) according to (\ref{phi_311}) we have, 
\begin{align*}
\frac{\mathbb P(\tilde X^1_{n+1} \in dz |\tilde X^1_{n}, \tilde X^2_{n},\tilde X^1_{n} =\tilde X^2_{n})}{\Lambda_{n,\tilde X^1_{n}, \tilde X^2_{n}}(dz)} 
= \frac{\mathbb P(\tilde X^2_{n+1} \in dz |\tilde X^1_{n}, \tilde X^2_{n},\tilde X^1_{n} =\tilde X^2_{n})}{\Lambda_{n,\tilde X^1_{n}, \tilde X^2_{n}}(dz)} =  p_n(\tilde X^1,z) =  p_n(\tilde X^2,z).
\end{align*}
Therefore, we have
\begin{align*}
\frac{\mathbb P(\tilde X^1_{n+1} \in dx^1 |\tilde X^1_{n}, \tilde X^2_{n})}{\Lambda_{n,\tilde X^1_{n}, \tilde X^2_{n}}(dx^1)} = \frac{\mathbb P(\tilde X^1_{n+1} \in dx^1 |\tilde X^1_{n}, \tilde X^2_{n})}{\Lambda_{n,\tilde X^1_{n}, \tilde X^2_{n}}(dx^1)}(1(\tilde X^1_{n} =\tilde X^2_{n}) + 1(\tilde X^1_{n} \not =\tilde X^2_{n}) 
 \\
= p_n(\tilde X^1_n,x^1) (1(\tilde X^1_{n} =\tilde X^2_{n}) + 1(\tilde X^1_{n} \not =\tilde X^2_{n})  = p_n(\tilde X^1_n,x^1),
\end{align*}
in all cases.

Due to all of these, each of the components $\tilde X^1_{n}$ and $\tilde X^2_{n}$ is a Markov process with the same generator as $X^1_{n}$ and $X^2_{n}$. (NB. The little calculus above is, of course, not the proof of the Markov property, which follows from the construction itself;  rather these formulae show how to understand the transition probability kernels of the chosen coupling algorithm.) Moreover, the following lemma holds true. 
\begin{Lemma}\label{lemma:2}
Let the random variables $\widetilde X^1_n$ and $\widetilde X^2_n$, for $n\in\mathbb{Z}_+$ be defined   by the following formulae:
\begin{align}\label{xin3}
\widetilde X^1_n:=\eta^1_n 1(\zeta_n=1)+\xi_n 1(\zeta_n=0), \quad 
\widetilde X^2_n:=\eta^2_n 1(\zeta_n=1)+\xi_n 1(\zeta_n=0).
\end{align}
Then 
\begin{equation}\label{l71}
\widetilde X^1_n\stackrel{d}{=}X^1_n, \;\;\widetilde
 X^2_n\stackrel{d}{=}X^2_n, \quad \mbox{for all $n\ge 0$,}
\end{equation}
which implies that the process $\widetilde X^1$ is equivalent to $X^1$, and the process $\widetilde X^2$ is equivalent to $X^2$ in distribution in the space of trajectories; in particular, each of them is a Markov process with the same generator as $X^1$.
Moreover, the couple $\tilde X_n:=\left(\widetilde X^1_n, \widetilde X^2_n\right)$, $n\ge 0$, is also  a  (non-homogeneous) Markov process, and 
\begin{equation*}\label{mpequi}
\left(\widetilde X^1_n\right)_{n\ge 0}\stackrel{d}{=}\left(X^1_n\right)_{n\ge 0}, \quad \& \quad 
\left(\widetilde X^2_n\right)_{n\ge 0}\stackrel{d}{=}\left(X^2_n\right)_{n\ge 0}.
\end{equation*}
Moreover, 
\begin{equation}\label{l723}
\widetilde X^1_n=\widetilde X^2_n, \quad \forall \; n\ge
n_0(\omega):=\inf\{k\ge0: \zeta_k=0\}, 
\end{equation}
and
\begin{equation}\label{estimate3}
\P_{\mu^1,\mu^2}(\widetilde X^1_n\neq \widetilde
 X^2_n)\le
\E_{\mu^1,\mu^2}\prod_{i=0}^{n-1}
 (1-\kappa_i(\eta^1_i,\eta^2_i)).
\end{equation}
\end{Lemma}

\medskip

\noindent
{\em Proof.} Let us first show (\ref{l723}) and (\ref{estimate3}).
As it follows from \eqref{process_eta} and (\ref{phi_411}),
\begin{align*}
&\displaystyle \mathbb P(\zeta_{n+1}=0|\zeta_n=0)=1,
 \\ \\
&\displaystyle \mathbb P(\zeta_{n+1}=0|\zeta_n=1,\eta^1_n=x^1,\eta^2_n=x^2)=\kappa_n(x^1,x^2).
\end{align*}
Indeed, given $\zeta_n=0$ we have, 
$$
\phi_{n,4}((x^1,x^2,x^3,0),u) =
1(x^4=0)\delta_0(u), 
$$
which implies that 
$$
\mathbb P(\zeta_{n+1}=0|\zeta_n=0)=1, 
$$
as required. As a consequence, 
we obtain (\ref{l723}).

Further, given $\zeta_n=1$ we have for any $(x^1,x^2,x^3)$, 
\begin{align*}
\phi_{n,4}((x^1,x^2,x^3,1),u):=\left(\delta_1(u)(1-\kappa_n(x^1,x^2))+ 
\delta_0(u)\kappa_n(x^1,x^2)\right), 
\end{align*}
which implies that 
$$
\mathbb P(\zeta_{n+1}=0|\zeta_n=1,\eta^1_n=x^1,\eta^2_n=x^2)=\kappa_n(x^1,x^2), 
$$
as required.
Hence, if two processes $\widetilde X^1$ and $\widetilde X^2$ are coupled at time $n$, then they will remain coupled at time $n+1$, and if they were not
coupled, then the coupling occurs with the (conditional) probability $\kappa(\eta^1_n,\eta^2_n)$; in this case the latter random variable equals $\kappa(X^1_n,X^2_n)$. So, 
from (\ref{xin3}) we deduce, 
\begin{align}\label{pred}
\P(\widetilde X^1_n\neq \widetilde X^2_n) \le  
\P (\zeta_n = 1) \equiv \P (\prod_{i=0}^{n}\zeta_i = 1) 
= \E \prod_{i=0}^{n}1(\zeta_i = 1) 
 \\ \nonumber \\  \nonumber 
= \E \prod_{i=0}^{n-1}1(\zeta_i = 1) \E (1(\zeta_n = 1)|{\cal F}_{n-1})
=  \E \prod_{i=0}^{n-1}1(\zeta_i = 1) (1- \kappa_{n-1}(\eta^1_{n-1}, \eta^2_{n-1})),
\end{align}
and further, using the conditional independence of the random variables  $\zeta_{n-1}, \eta^1_{n-1}, \eta^2_{n-1}$ given ${\cal F}_{n-2}$ by virtue of the construction (\ref{process_eta2}), we continue
\begin{align*}
\P(\widetilde X^1_n\neq \widetilde X^2_n) \le  
\E \prod_{i=0}^{n-2}1(\zeta_i = 1) 1(\zeta_{n-1} = 1)(1- \kappa_{n-1}(\eta^1_{n-1}, \eta^2_{n-1})) 
\\\\
= \E \prod_{i=0}^{n-2}1(\zeta_i = 1) \E (1(\zeta_{n-1} = 1)(1- \kappa_{n-1}(\eta^1_{n-1}, \eta^2_{n-1})) | {\cal F}_{n-2})
 \\\\
= \E \prod_{i=0}^{n-2}1(\zeta_i = 1) \E (1(\zeta_{n-1} = 1)| {\cal F}_{n-2}) \E((1- \kappa_{n-1}(\eta^1_{n-1}, \eta^2_{n-1})) | {\cal F}_{n-2})
\\\\
= \E \prod_{i=0}^{n-2}1(\zeta_i = 1) (1- \kappa_{n-1}(\eta^1_{n-2}, \eta^2_{n-2}))  \E((1- \kappa_{n-1}(\eta^1_{n-1}, \eta^2_{n-1})) | {\cal F}_{n-2}) 
\\\\
= \E \prod_{i=0}^{n-2}1(\zeta_i = 1) (1- \kappa_{n-1}(\eta^1_{n-2}, \eta^2_{n-2}))  (1- \kappa_{n-1}(\eta^1_{n-1}, \eta^2_{n-1})),   
\end{align*}
and so on by induction, which finally gives us the desired bound  (\ref{estimate3}).

\medskip

Further, for the pair $(\tilde X^1, \tilde X^2)$ we have its transition density with respect to $\Lambda_{n,\tilde X^1_{n}, \tilde X^2_{n}}$,
\begin{align*}
\frac{\mathbb P(\tilde X^1_{n+1} \in dx^1, \tilde X^2_{n+1} \in dx^2 |\tilde X^1_{n}, \tilde X^2_{n})}{\Lambda_{n,\tilde X^1_{n}, \tilde X^2_{n}}(dx^1)}
 \\
= 1(\tilde X^1_{n} \not =\tilde X^2_{n}) 
\frac{\mathbb P(\tilde X^1_{n+1} \in dx^1, \tilde X^2_{n+1} \in dx^2 |\tilde X^1_{n}, \tilde X^2_{n})}{\Lambda_{n,\tilde X^1_{n}, \tilde X^2_{n}}(dx^1)}
 \\
+ 1(\tilde X^1_{n} =\tilde X^2_{n})  \frac{\mathbb P(\tilde X^1_{n+1} \in dx^1, \tilde X^2_{n+1} \in dx^2 |\tilde X^1_{n}, \tilde X^2_{n})}{\Lambda_{n,\tilde X^1_{n}, \tilde X^2_{n}}(dx^1)}
 \\
= 1(\tilde X^1_{n} \not =\tilde X^2_{n}) p_n(\tilde X^1_{n}, x^1) p_n(\tilde X^2_{n},x^2)  
+1(\tilde X^1_{n} =\tilde X^2_{n}) p_n(\tilde X^1_{n}, x^1)\delta(x_1-x_2).
\end{align*}
Until the first moment where $\zeta = 0$ both trajectories $\tilde X^1$ and  $\tilde X^2$ evolve independently, and the pair $(\tilde X^1, \tilde X^2)$ is a strong Markov process until this stopping time. After this stopping time they are equal and each of them remains a strong Markov process with the same transition kernel. This justifies (\ref{l71}) and (\ref{l723}), and the claim that the pair $(\tilde X^1_n, \tilde X^2_n, n\ge 0)$ is Markov and strong Markov. 
The Lemma \ref{lemma:2} is proved.

\subsection{Operators $V$ and $\hat V$ in the non-homogeneous case}\label{sec:Vr}
In the non-homogeneous situation it does not seem possible to improve or to simplify the bound (\ref{estimate3}) analogously to the simplification (\ref{newrate12}) of the estimate (\ref{estimate}) in the homogeneous case, without additional assumptions on the structure of the transition kernels. Two special cases will be discussed in what follows. 

Let us introduce the operators $V_t$ acting on (bounded, Borel measurable) functions $h$ on the space $S^2 := S\times S$ as follows: for $x=(x^1, x^2)\in S^2$,
\begin{equation}\label{V3}
V_th(x) := (1-\kappa_t(x^1,x^2)) \mathbb E_{t,x^1,x^2}h(\tilde X_{t+1}) \equiv \exp(\psi_t(x))\mathbb E_{t,x^1,x^2}h(\tilde X_{t+1}),  
\end{equation}
where in the last expression $\psi_t(x):= \ln (1-\kappa_t(x^1,x^2))$ (assume $\ln 0 = -\infty$); recall that 
$\tilde X_n = (\tilde X_n^1, \tilde X^2_n)$. Note that on the diagonal $x=(x^1,x^2) : x^1=x^2$ we have 
$$
V_th(x) = (1-\kappa_t(x^1,x^1))\mathbb E_{t,x^1,x^2}h(\tilde X_{t+1}) = 0, 
$$
since $\kappa_t(x^1,x^1) = 1$ for any $x^1$. Hence, similarly to the homogeneous case, it makes sense to either consider the functions $h$ on $S^2$ vanishing on the diagonal $\text{diag}(S^2)= (x =  (x^1,x^1)\in S^2)$, or, equivalently, to reduce the operator itself to functions defined on 
$$
\hat S^2 := S^2 \setminus \text{diag}(S^2), 
$$
that is, to define for $x=(x^1,x^2)\in \hat S^2$ and for functions $\hat h: \hat S^2 \to \mathbb R$, 
$$
\hat V_t\hat h(x):= (1-\kappa_t(x^1,x^2))\mathbb E_{t,x^1,x^2}\hat h(\tilde X_{t+1}). 
$$

The estimate (\ref{estimate3}) can be rewritten via the operators $V_t$, or, equivalently,  via $\hat V_t$ as follows: 
\begin{align*}
\!\!\P_{\mu^1,\mu^2}(\widetilde X^1_n\neq \widetilde X^2_n)\!\le\! 
\int 
\prod_{i=0}^{n-1} \hat V_i{\bf 1}(x^1,x^2)1(x^1\!\not =\!x^2)\mu^1(dx^1) \mu^2(dx^2).
\end{align*}

~

\begin{Remark}
Note that in may be more convenient in examples to use the ``preliminary'' bound
\begin{align*}
\P(\widetilde X^1_n\neq \widetilde X^2_n) \le  
\P (\zeta_n = 1) \equiv \E \left(\prod_{i=0}^{n}1(\tilde X^1_i \neq \tilde X^2_i)\right), 
\end{align*}
equivalent to (\ref{pred}).

\end{Remark}

\begin{Theorem}\label{lastthm3}
In all cases
\begin{align}\label{newrate13}
\limsup\limits_{n\to\infty} \frac1n \ln \| \P_{\mu^1}(n,\cdot) - \P_{\mu^2}(n,\cdot)\|_{TV}
  \nonumber \\\\ \nonumber
\!\le\! \limsup\limits_{n\to\infty} \frac1{n} \ln \int 
\prod_{i=0}^{n-1} \hat V_i{\bf 1}(x^1,x^2)1(x^1\!\not =\!x^2)\mu^1(dx^1) \mu^2(dx^2).
\end{align}
\end{Theorem}

{\em Proof.} We have, 

\begin{align*}
\limsup\limits_{n\to\infty} \frac1n \ln \|  \P_{\mu^1}(n,\cdot) - \P_{\mu^2}(n,\cdot)\|_{TV} 
\le \limsup\limits_{n\to\infty} \frac1n \ln (2\P_{\mu^1,\mu^2}(\widetilde X^1_n\neq \widetilde X^2_n))
 \\\\
= \limsup\limits_{n\to\infty} \frac1n \ln (2 \int \P_{x^1,x^2}(\widetilde X^1_n\neq \widetilde X^2_n)\mu^1(dx^1)\mu^2(dx^2))
  \nonumber \\\\ \nonumber
\!\le\! \int 
\prod_{i=0}^{n-1} \hat V_i{\bf 1}(x^1,x^2)1(x^1\!\not =\!x^2)\mu^1(dx^1) \mu^2(dx^2),  
\end{align*}
as required. The theorem is proved.

\begin{Corollary}
In all cases, 
\begin{align}\label{newrate13a}
\limsup\limits_{n\to\infty} \frac1n \ln \| \P_{\mu^1}(n,\cdot) - \P_{\mu^2}(n,\cdot)\|_{TV}
\!\le\! \limsup\limits_{n\to\infty} \frac1{n} \ln \| \prod_{i=1}^{n} \hat V_i \|.
\end{align}
The limit is uniform with respect to the initial measures $\mu^1, \mu^2$.
\end{Corollary}

\medskip

\noindent
As it was noted earlier, in general there seems to be no way to apply the approach based on the spectral radius, so the bounds (\ref{newrate13}) and (\ref{newrate13a}) look as the best possibility. However, there are two special cases where a bit more can be added. The first case is a periodic transition kernel $\P_{t,x}(t+1,dy)$. Let there exist $T\ge 1$ such that 
\begin{equation}\label{T}
\P_{t+T,x}(t+T+1,dy) = \P_{t,x}(t+1,dy), \quad \forall \, t\ge 0.
\end{equation}
Clearly, this implies 
$$
\kappa_t(x^1,x^2) = \kappa_{t+T}(x^1,x^2), \quad \forall t\ge 0. 
$$
Consider the operators 
$$
V_{}^{(T)}h(x) = \mathbb E_{0,x^1,x^2}(\prod_{i=0}^{T-1} (1-\kappa_i(\tilde X_{i}))) h(\tilde X_{T}), \quad x\in S^2, 
$$
acting on functions on $S^2$, and 
$$
\hat V_{}^{(T)}h(x) = \mathbb E_{0,x^1,x^2}(\prod_{i=0}^{T-1} (1-\kappa_i(\tilde X_{i}))) h(\tilde X_{T}), \quad x\in \hat S^2, 
$$
acting on functions on $\hat S^2$.
\begin{Theorem}\label{lastthm4}
Under the assumption (\ref{T}), 
\begin{align}\label{newrate13d}
\limsup\limits_{n\to\infty} \frac1n \ln \| \P_{\mu^1}(n,\cdot) - \P_{\mu^2}(n,\cdot)\|_{TV}
  \nonumber \\\\ \nonumber
\!\le\! \limsup\limits_{n\to\infty} \frac{1}{n} \ln \int 
(\hat V^{(T)})^{[n/T]}{\bf 1}(x^1,x^2)1(x^1\!\not =\!x^2)\mu^1(dx^1) \mu^2(dx^2)
\le \frac{1}{T}\ln r(\hat V^{(T)}).
\end{align}
The limit is uniform with respect to the initial measures $\mu^1, \mu^2$.
\end{Theorem}
{\em Proof} follows straightforwardly from the bounds (\ref{newrate13}), (\ref{newrate13a}) and Gelfand's formula for the spectral radius.

\medskip

The second special case is the situation of small non-homogeneous perturbations of the {\em homogeneous} kernel $\bar \P_{x^1}(1,dy)$. Let us assume that all measures $\bar \P_{x^1}(1,dy)$ are absolutely continuous with respect to the same dominating measures as $\P_{t,x^1}(1,dy)$ (for all $t$), and that the densities of the latter measures do not differ too much in a certain sense from the densities of the measures $\bar \P_{x^1}(1,dy)$. One version of this situation is considered in what follows under the conditions  (\ref{ppt3})--(\ref{ppt5}). 

The analogue of the homogeneous bound  (\ref{newrate12}) in the non-homogeneous situation will require a certain modification of the auxiliary process $(\eta^1_n, \eta^2_n, \zeta_n, \xi_n)$, as well as a new ergodic coefficient and a new homogeneous operator, the analogue of the operator $V$ from the previous section. Let
\begin{equation*}
\bar \phi(x,y):=\bar \phi_{1}(x,y^1)\bar \phi_{2}(x,y^2)\bar \phi_{3}(x,y^3) \bar \phi_{4}(x,y^4),
\end{equation*}
where $x=(x^1,x^2,x^3,x^4)$, $y=(y^1,y^2,y^3,y^4)$; and, if  
 $0<\bar \kappa(x^1,x^2)<1$ then 
\begin{align*}
\bar \phi_{1}(x,u):=\frac{\bar p(x^1,u)- \bar p(x^1,
u)\wedge \bar p(x^2,u)}{1- \bar \kappa(x^1,x^2)}, \quad  
\bar \phi_{2}(x,u):=\frac{\bar p(x^2,u) - \bar p(x^1,u)\wedge \bar p(x^2,u)}{1- \bar \kappa(x^1,x^2)}, 
 \\\nonumber\\
\bar \phi_{3}(x,u)\!:=\!1(x^4\!=\!1)\frac{\bar p(x^1,u)\wedge
 \bar p(x^2,u)}{\bar \kappa(x^1,x^2)}\!+\!1(x^4\!=\!0)\bar p(x^3,u), 
 \\\nonumber\\ 
\bar \phi_{4}(x,u):=1(x^4=1)\left(\delta_1(u)(1-\bar \kappa(x^1,x^2)) + 
\delta_0(u) \bar \kappa(x^1,x^2)\right) +1(x^4=0)\delta_0(u); 
\end{align*}
and in the case where $\kappa(x^1,x^2)=1$, or if $\kappa(x^1,x^2)=0$, the transition densities will be defined similarly to the homogeneous case, see (\ref{after0}) and (\ref{after}). 

Let us construct a (homogeneous) Markov process $(\bar \eta^1_n, \bar \eta^2_n, \bar \xi_n, \bar \zeta_n)$ similarly to the process $(\eta^1_n, \eta^2_n, \xi_n, \zeta_n)$ from the section \ref{Sec22} by using 
these transition densities, and let 
\begin{align}\label{xin2}
\bar X^1_n:=\bar \eta^1_n 1(\bar \zeta_n=1) + \bar \xi_n 1(\bar \zeta_n=0), \quad 
\bar  X^2_n:=\bar \eta^2_n 1(\bar \zeta_n=1)+\bar \xi_n 1(\bar \zeta_n=0).
\end{align}
(NB: the process $(\bar X^1_n, \bar X^2_n)$ is the analogue of the processes $(\tilde X^1_n, \tilde X^2_n)$ both in the homogeneous and in the non-homogeneous case; however, the notation $(\tilde X^1_n, \tilde X^2_n)$ in our situation is already in use; so, we have to introduce a new one.)

Assume that there are values  $\epsilon, \delta>0$ such that for all $t,x$ the non-homogeneous kernels $\P_{t,x}$ are abolutely continuous with respect to each other and with respect to $\bar  \P_{x}$: $\P_{t,x}(t+1,dy) \sim \bar  \P_{x}(1,dy)$; more than that, for all $n$
\begin{equation}\label{ppt3}
(1+\epsilon)^{-1} \le 
\frac{\kappa_{n}(x^1,x^2)}{\bar\kappa_{}(x^1,x^2)}, 
\quad 
\frac{1-\kappa_{n}(x^1,x^2)}{1-\bar\kappa_{}(x^1,x^2)} \le 1+\epsilon, 
\end{equation} 
and also, uniformly in $t\ge 0$, 
\begin{align}
\phi_{t,1}(x,u) 
\le (1+\epsilon)
\bar \phi_{1}(x,u), 
\label{ppt4}
 \\\nonumber\\
\phi_{t,3}(x,u) 
\displaystyle \le (1+\epsilon)\bar \phi_{3}(x,u). 
\label{ppt5}
\end{align}
If the inequality (\ref{ppt4}) holds true for the pair  $\phi_{t,1}, \bar \phi_{1}$, then it is also valid for the couple $\phi_{t,2}, \bar \phi_{2}$. Note that in the case of the finite state space $S$ all the conditions (\ref{ppt3})--(\ref{ppt5}) are, in principle, realisable. Indeed, for such  $S$ all nonzero -- hence, positive and bounded away from zero -- values in these expressions may only differ no more than by a multiplier $(1+\epsilon)$ (possibly by $(1+C\epsilon)$ with some $C>0$, which just leads to a redefinition of the small parameter) in both directions from the non-perturbed ones  under small perturbations; moreover, for the finite $S$ there are just the finite number of these expressions; hence, all their strictly positive linear combinations have this property, too. 

\begin{Proposition}\label{prononhomo}
Let only the first inequality from  (\ref{ppt3}) hold true for some homogeneous transition kernel $\bar \P$, which satisfies all the assumptions of theorem \ref{lastthm2}. Then Markov -- Dobrushin's bound  (\ref{exp_bd3}) is transformed as follows:

\begin{equation}\label{exp_bd34}
 \sup_{x}\sup_{A\subset S} |\P_{0,x}(n,A) - \P_{0,x'}(n,A)| \le (1- (1+\epsilon)^{-1} \bar\kappa)^n. 
\end{equation}
If the second condition in (\ref{ppt3}) is satisfied then the following bound holds  true, 
\begin{equation}\label{exp_bd344}
 \sup_{x}\sup_{A\subset S} |\P_{0,x}(n,A) - \P_{0,x'}(n,A)| \le (1+\epsilon)^{n} (1- \bar\kappa)^n, 
\end{equation}
where $\bar \kappa$ is the ``homogeneous'' characteristics of the kernel  $\bar \P_{x}(1,dy)$, defined according to  (\ref{betterMD2}). 
\end{Proposition}
Naturally, the bounds (\ref{exp_bd34}) and (\ref{exp_bd344}) make sense under the condition 
$\bar \kappa >0$.  

~

{\em Proof.}
We have, 
\begin{align*}
\kappa_t(x^1,x^2) 
\ge (1+\epsilon)^{-1} \bar\kappa(x^1,x^2) 
\ge (1+\epsilon)^{-1} \inf_{x^1,x^2}\bar\kappa(x^1,x^2) 
=(1+\epsilon)^{-1} \bar\kappa.
\end{align*}
So, 
\begin{align*}
\!\!1\!-\!\kappa_t(x^1,x^2)  \!\le \!1\!-\! (1+\epsilon)^{-1} \bar\kappa(x^1,x^2) \le \!1\!-\! (1+\epsilon)^{-1} \bar\kappa, \quad
1\!-\!\kappa_t \!\le \!1\!-\! (1+\epsilon)^{-1} \bar\kappa.
\end{align*}
and, hence, 
(see (\ref{estimate3}))
\begin{align}\label{newrate13cc}
\sup_{x}\sup_{A\subset S} |\P_{0,x}(n,A) - \P_{0,x'}(n,A)| \le \prod_{t=0}^{n-1} (1-\kappa_t)^{}
\le (1- (1+\epsilon)^{-1} \bar\kappa)^n, 
\end{align}
which shows that (\ref{exp_bd34}) is valid. the bound (\ref{exp_bd344}) follows straightforwardly from the second part of the condition  (\ref{ppt3}). Proposition \ref{prononhomo} follows.

\begin{Remark}
The first inequality in (\ref{newrate13cc}) can be found in \cite[formula  (29)]{Kolm}. In essence, this estimate belongs to Markov himself, although he obtained it just for the case of the finite homogeneous Markov chains. In the present paper the ineqaility (\ref{exp_bd34}) is provided only for the comparison to the following result of theorem \ref{lastthm5}, and its elementary derivation from the stronger bound (\ref{estimate3}) is given purely for the completeness of the presentation. 
\end{Remark}

~

The operator $\bar V$ is introduced similarly to the operator $V$ in the homogeneous case via the process $\bar X=(\bar X^1, \bar X^2)$, $\bar V^{}h(x^1,x^2):= 
\mathbb E_{x^1,x^2}h(\bar X_{1}) 1(\bar X_{1}^1\neq \bar X_{1}^2)  
$, or, equivalently, via the process $(\bar \eta^1, \bar \eta^2, \bar \xi, \bar \zeta)$ by the formula

\begin{equation}\label{newV}
\bar V^{}h(x^1,x^2):= 
\mathbb E_{x^1,x^2}h(\bar X_{1}) 1(\bar \zeta^{}_1=1).  
\end{equation}
Note that $\bar V$ does depend on the parameter $\epsilon$, which is not shown in the notation. The symbol $r(\bar  V^{})$ will be used for the spectral radius of $\bar  V^{}$. The homogeneous analogue of the operator $\hat V_t$ on the state space $\hat S^2$ without the diagonal will be denoted by  $\check V$: for  $x = (x^1,x^2) \in \hat S^2$ and $h: \hat S^2 \mapsto \mathbb R$, 
\begin{equation}\label{vcheck1}
\check  V h(x) := (1 - (1+\epsilon)^{-1}\bar \kappa(x^1,x^2)) \mathbb E_x h(\bar X^1_1, \bar X^2_1) 1(\bar X^1_1 \neq \bar X^2_1).
\end{equation}
(This operator may be also introduced by the formula
\begin{equation}\label{vcheck2}
\check  V h(x) := (1+\epsilon)(1 - \bar \kappa(x^1,x^2)) \mathbb E_x h(\bar X^1_1, \bar X^2_1) 1(\bar X^1_1 \neq \bar X^2_1), 
\end{equation}
then the statement of theorem  \ref{lastthm5} will change slightly.)

~

Let us note certain properties of the new process  $(\bar X^1, \bar X^2)$.
\begin{Lemma}\label{lenohomo}
Let assumptions (\ref{ppt3})--(\ref{ppt5}) be satisfied. Then 
$$
\P_t(x,x',dy,dy') \le (1+\epsilon)^2\bar \P(x,x',dy,dy').
$$
\end{Lemma}

~

\noindent
{\em Proof} for $x\neq x'$ follows from the definitions (\ref{xin3}) and (\ref{xin2}), and from the assumptions  (\ref{ppt3})--(\ref{ppt5}), which, in particular, imply the following inequalities for the densities 
(here $x=(x^1,x^2,x^3,x^4)$): 
\begin{align*}
p^{\tilde X^1_n|\tilde X^1_{n-1}, \tilde X^2_{n-1}}(u)|_{\tilde X^1_{n-1}=\tilde X^2_{n-1}=x^3}=\varphi_{n,3}(*,*,x^3,1;u) \le (1+\epsilon)\bar \varphi_{}(*,*,x^3,1;u) 
 \\\\
= (1+\epsilon) p^{\bar X^1_n|\bar X^1_{n-1}, \bar X^2_{n-1}}(u)|_{\bar X^1_{n-1}= \bar X^2_{n-1}=x^3}
\end{align*}
according to condition (\ref{ppt5}), and 
\begin{align*}
p^{\tilde X^1_n|\tilde X^1_{n-1}, \tilde X^2_{n-1}}(u)|_{\tilde X^1_{n-1}=x^1 \neq x^2=\tilde X^2_{n-1}}
 \\\\
= \mathbb E_{}(\phi_{n,1}(x^1,x^2, *,0;u)|_{x^1 \neq x^2}
(1-\kappa_{n}(\eta^1_{n},\eta^2_{n}))|X^1_{n-1}, X^2_{n-1}, X^1_{n-1}\neq X^2_{n-1})|_{X^1_{n-1}=x^1 \neq x^2=X^2_{n-1}}
 \\\\
+ \mathbb E_{}(\phi_{n,3}(x^1,x^2, *,0;u)\kappa_n(\eta^1_{n},\eta^2_{n})|X^1_{n-1}, X^2_{n-1}, X^1_{n-1}\neq  X^2_{n-1})|_{X^1_{n-1}=x^1, X^2_{n-1}=x^2}
 \\\\
\le (1+\epsilon)^2\mathbb E_{}(\phi_{n,1}(x^1,x^2, *,0;u)|_{x^1 \neq x^2}
(1-\bar\kappa_{}(\bar \eta^1_{n}, \bar \eta^2_{n}))|\bar X^1_{n-1}, \bar X^2_{n-1}, \bar X^1_{n-1}\neq \bar X^2_{n-1})|_{\bar X^1_{n-1}=x^1 \neq x^2= \bar X^2_{n-1}}
 \\\\
+ (1+\epsilon)^2\mathbb E_{}(\bar \phi_{3}(x^1,x^2, *,0;u)\bar \kappa(\bar\eta^1_{n},\bar\eta^2_{n})|\bar X^1_{n-1}, \bar X^2_{n-1}, \bar X^1_{n-1}\neq  \bar X^2_{n-1})|_{\bar X^1_{n-1}=x^1 \neq x^2 =\bar X^2_{n-1}}
 \\\\
=(1+\epsilon)^2 p^{\bar X^1_n|\bar X^1_{n-1}, \bar X^2_{n-1}}(u)|_{\bar X^1_{n-1}=x^1 \neq x^2=\bar X^2_{n-1}}
\end{align*}
due to the conditions (\ref{ppt3})--(\ref{ppt4}), which proves the lemma.

~

Denote for what follows
$$
\delta := (1+\epsilon)^2 - 1.
$$


\begin{Theorem}\label{lastthm5}
Under the assumptions (\ref{ppt3})--(\ref{ppt5}), the bound 
\begin{align}\label{newrate13c}
\limsup\limits_{n\to\infty} \frac1n \ln \| \P_{\mu^1}(n,\cdot) - \P_{\mu^2}(n,\cdot)\|_{TV}
\le \ln ((1+\delta)^{}  r(\bar  V^{})) 
\end{align}
holds true. 
The upper limit in the left hand side here is uniform with respect to the initial distributions $\mu^1, \mu^2$. 
\end{Theorem}
Clearly, the bound (\ref{newrate13c}) itself makes sense if 
\begin{equation}\label{sense1}
(1+\delta)^{}r(\bar  V^{}) < 1, 
\end{equation}
and it can be more efficient than the general bound (\ref{exp_bd34}) in the case of 
\begin{equation}\label{sense2}
(1+\delta)^{} r(\bar  V^{}) < 1- (1+\epsilon)^{-1} \bar\kappa.
\end{equation} 
The latter inequality may be satisfied in the case of a small  $\epsilon$ if  $\bar\kappa \approx 0$ (or even if $\bar\kappa = 0$). Similarly, the bound  (\ref{newrate13c}) may be compared to the inequality  (\ref{exp_bd344}).

~

\noindent
{\em Proof.}
We have for $h\ge 0$, 
\begin{align*}
V_th(x) = (1-\kappa_t(x^1,x^2))\mathbb E_{t,x^1,x^2}h(\tilde X_{t+1})
= (1-\kappa_t(x^1,x^2))\int h(y^1,y^2) \mathbb P_t(x^1,x^2,dy^1,dy^2)
 \\\\
\le (1-\kappa_t(x^1,x^2))(1+\epsilon)^2\int h(y^1,y^2) \bar {\mathbb P}(x^1,x^2,dy^1,dy^2)
 \\\\
\le (1- (1+\epsilon)^{-1} \bar\kappa(x^1,x^2))(1+\epsilon)^2\int h(y^1,y^2) \bar {\mathbb P}(x^1,x^2,dy^1,dy^2).
\end{align*}
Hence, due to theorem \ref{lastthm3} and by virtue of the fact that the multipliers $(1-\kappa_i(\tilde x_{i}))$ are non-negative, we estimate: 
\begin{align*}
\int 
\left((\prod_{i=0}^{n-1} V_i) {\bf 1}\right)(x^1,x^2)\mu^1(dx^1) \mu^2(dx^2) 
 \\\\
= \int (\prod_{i=0}^{n-1} (1-\kappa_i(\tilde x_{i})) \P_i(x_{i}, x'_{i}, d x_{i+1},x'_{i+1}) \mu^1(dx^1) \mu^2(dx^2) 
 \\\\
\le (1+\epsilon)^{2n}  \int (\prod_{i=0}^{n-1} (1-\kappa_i(x_{i}, x'_i)) \bar \P(x_{i}, x'_{i}, d x_{i+1}, d x'_{i+1}) \mu^1(dx^1) \mu^2(dx^2)
 \\\\
\le (1+\epsilon)^{2n}  \int (\prod_{i=0}^{n-1}(1- (1+\delta)^{-1} \bar\kappa(x_{i}, x'_i)) \bar \P(x_{i}, x'_{i}, d x_{i+1}, d x'_{i+1}) \mu^1(dx^1) \mu^2(dx^2)
 \\\\
\le (1+\epsilon)^{2n} \int 
\left((\bar  V^{})^n {\bf 1}\right)(x^1,x^2) 
\mu^1(dx^1) \mu^2(dx^2).
\end{align*}
Now the ``homogeneous'' theorem \ref{lastthmm2m} guarantees the desired asymptotic bound 
\begin{align*}
\limsup\limits_{n\to\infty} \frac1n \ln \| \P_{\mu^1}(n,\cdot) - \P_{\mu^2}(n,\cdot)\|_{TV}
 \\\\
\le \limsup\limits_{n\to\infty} \frac1n \ln (1+\epsilon)^{2n} \int 
\left((\bar  V^{})^n{\bf 1}\right)(x^1,x^2)\mu^1(dx^1) \mu^2(dx^2)
 \\\\
\le \ln (1+\epsilon)^2 + \ln r(\bar  V^{}), 
\end{align*}
as required. Theorem \ref{lastthm5} is proved.

\begin{Remark}
Naturally, the classes of non-homogeneous Markov chains described in theorems \ref{lastthm4} and \ref{lastthm5} (and in proposition \ref{prononhomo}) intersect. In the intersection the bounds (\ref{newrate13d}), and (\ref{exp_bd34}), and (\ref{newrate13c}) are applicable.
\end{Remark}

\section{Examples}\label{sec:ex}
In all examples in this section $|S|<\infty$. Let us briefly discuss some specifications implied by this convention. 
In this situation for any irreducible non-cyclic homogeneous Markov chain which possesses a {\em spectral decomposition} -- which means that right as well as left eigenvectors form bases in the corresponding linear spaces (both equivalent to ${\mathbb R}^{|S|}$), see \cite[Chapter 4]{Karlin} -- the rate of convergence towards the stationary distribution is equivalent to $C\gamma^n$ with some constant $C$, where $\gamma$ is  the spectral gap of the transition matrix (see \cite{KLS}), which in the case of stochastic matrices coincides with the maximal modulus of the rest of the spectrum of the transition matrix  (see, e.g., \cite{Gantmacher,Karlin}). 

~

Also, for finite irreducible non-cyclic homogeneous stochastic matrices there is a general formula for the difference $p_{ij}^{(n)} - \pi_j$ (where $\pi$ denotes the unique stationary distribution), 
which implies the same rate $C\gamma^n$ with some constant $C$ computable on the basis of the characteristic polynomial of the matrix ${\cal P}$, see \cite[Chapter XIII, (96)]{Gantmacher}. 
Hence, at least, in the case $|S|<\infty$ we obtain the inequality
\[
\gamma \le r(V). 
\]
Yet, for infinite $|S|$ or even for finite but large values of $|S|$ the task of finding $\gamma$ may be practically a bit more difficult than computing the spectral radius $r(V)$ for computing which there are quite a few recipes, and even more so for more general state spaces where the spectrum of the transition operator may have a more complicated nature. However, in this section  we only consider examples for finite state spaces. See, for example, \cite{KLS} for some useful practical approaches for computing spectral radii of general operators. 

~

What's more, for reversible Markov chains the following precise bound 
\begin{equation}\label{dse3}
\|P_i(n,\cdot) - \mu\|_{TV} \le \left(\frac{1-\mu(i)}{2\mu(i)}\right)^{1/2} \gamma^n
\end{equation}
holds true  \cite[Proposition 3]{Stroock-Diaconis}, where $\gamma$ is again a spectral gap of the matrix $\cal P$. 

~

In the simplest case of $|S|=2$ our matrix $V$ -- as well as $\cal P$ -- is of dimension $2\times 2$. So all (both) eigenvalues can be computed by hand explicitly. Their comparison shows that the following is true.
\begin{Proposition}
Let $0<a,b<1$, and
\[
{\cal P} = \left(
\begin{array}{l l}
a & 1-a \\
1-b & b
\end{array}
\right)\,.
\]
Then in all cases $r=1-\kappa  \ge |\lambda_{2}|$, 
where $\lambda_{2}$ is the smaller eigenvalue of the original transition matrix $\cal{P}$.

\end{Proposition}

\begin{proof}

\begin{itemize}

	\item Case 1) when $a, b \le 1/2$. 

Firstly, the second eigenvalue of this transition matrix has the absolute value $|\lambda_{2}|=|1-(a+b)|$.

In this case $\kappa=a+b$, so $1-\kappa=1-(a+b)$. First we compute the transition matrix for the coupled process. Here, we will write I, II for states $(1,2), (2,1)$, respectively and $\mathcal{P}_{c}$ will denote the corresponding transition probability matrix. We have that 
$$
p_{I, I}=p_{\{(1, 2), (1,2)\}}=\frac{a-\min(a,1-b)}{1-\kappa}*\frac{b-\min(b, 1-a)}{1-\kappa},
$$

and also 
$$
p_{I, II}=\frac{1-a-\min((1-a), b)}{1-\kappa}*\frac{(1-b)-\min((1-b), a)}{1-\kappa},
$$

$$
p_{II, I}=\frac{1-b-\min((1-b), a)}{1-\kappa}*\frac{(1-a)-\min((1-a), b)}{1-\kappa}.
$$

$$
p_{II, II}=\frac{b-\min((1-a), b)}{1-\kappa}*\frac{a-\min((1-b), a)}{1-\kappa}.
$$

This boils down to $p_{I,I}=p_{II, II}=0$ and $p_{I, II}=p_{II, I}=1$. This is because for $a, b \le 1/2$ we have that $a-\min(a, 1-b)=a-a=0$, $1-\min(1-a, b)=1-(a+b)=1-\kappa$, $1-b-\min(1-b, a)=1-(a+b)=1-\kappa$.

Here the spectral radius of $(1-\kappa)\mathcal{P}_{c}$ also has the modulus of $|1-(a+b)|$, hence coinciding with the modulus of the second eigenvalue of the original transition matrix, as we claimed. 

	\item Case (2a) is when $a \le 1/2 \le b$, $a \le 1-b$:

In this case we obtain $p_{I, I}=p_{II, II}=0, p_{I, II}=p_{II, I}=1$, because now $\min(a, 1-b)=a$ and $\min(1-a, b)=b$. 

	\item  Case (2b): when $a \le 1/2 \le b$, $a > 1-b$:

In this case we obtain $p_{I, I}=p_{II, II}=1, p_{I, II}=p_{II, I}=0$, because now $\min(a, 1-b)=1-b$ and $min(b, 1-a)=1-a$. 

	\item Case 3: when $a, b \ge 1/2$:

In this case we obtain $p_{I, I}=p_{II, II}=1, p_{I, II}=p_{II, I}=0$, because now $\min(a, 1-b)=1-b$ and $\min(b, 1-a)=1-a$. The repeated eigenvalues of $(1-\kappa)\mathcal{P}$ are $\lambda=1-\kappa=1-(a+b)$, so $r=1-(a+b)$.
\end{itemize}
So, in all cases the result follows.
\end{proof}

\medskip

\medskip

\begin{Ex}[The three approaches are similar]
Consider the MP with the state space $S = \{1,2\}$ and the original transition matrix
\[
{\cal P} = \left(
\begin{array}{l l}
0.65 & 0.35 \\
0.35 & 0.65
\end{array}
\right)\,.
\]
We have, $\hat S^2 = \{(1,2), (2,1)\}$, and 
\[
V = (1-0.7) \times 
\left(
\begin{array}{l l}
0.65 - 0.35 & 0 \\
0 & 0.65 - 0.35
\end{array}
\right)\times \frac{1}{0.3} = 0.3 \times 
\left(
\begin{array}{l l}
1 & 0 \\
0 & 1
\end{array}
\right)\,.
\]
So we compute, 
\[
1-\kappa = 0.3   =  r(V), 
\]
as expected. 
Also, this MC is reversible and we can compare our bound with (\ref{dse3}):
\[
\|P_i(n,\cdot) - \mu\|_{TV} \le \left(\frac{1-\pi(i)}{2\pi(i)}\right)^{1/2} \lambda_2^n. 
\]
We have, 
\[
\lambda_2 = 0.3 = r(V), 
\]
again as expected since here $\lambda_2=\gamma$ and $\lambda_2 = 1-\kappa$. So, in this simple example our approach provides practically the same (asymptotic) bound for the rate of convergence as the classical Markov-Dobrushin's and (also asymptotically) the same as the spectral one based on the second eigenvalue of the original transition matrix. 

Note that this example is so simple that does not require any code: all computations were performed by hand. However, we double checked it with the code too.

\end{Ex}
\begin{Ex}[The three approaches are similar]
Consider the MP with the state space $S = \{1,2\}$ and transition matrix
\[
{\cal P} = \left(
\begin{array}{l l}
0.8 & 0.2 \\
0.4 & 0.6
\end{array}
\right)\,.
\]
This matrix has eigenvalues $\lambda_{1}=1, \lambda_{2}=0.4$.

We have, $\hat S^2 = \{(1,2), (2,1)\}$, and the normalized transition matrix of the ``doubled'' process $(\eta^1_n, \eta^2_n)$ 
reads
\[
(1-\kappa)^{-1} V 
=  \left(
\begin{array}{l l}
0.48/0.56 & 0.08/0.56 \\
0.08/0.56 & 0.48/0.56
\end{array}
\right).
\]
We compute, 
\[
\kappa = \kappa(1,2) = 0.6,\quad 1-\kappa = 0.4   =  r(V) = |\lambda_{2}|, 
\]
as expected. We have also double checked it with the code. 
\end{Ex}

\begin{Ex}[The new approach is similar to the MD; the eigenvalue one is better]
Consider the MP with the state space ${\cal S} = \{1,2,3\}$ and transition matrix
\[
{\cal P} = \left(
\begin{array}{l l l}
0 & 0.3 & 0.7  \\
0.3 & 0.7 & 0 \\
0.7 & 0 & 0.3
\end{array}
\right)\,.
\]
Here the matrix is doubly stochastic, so $(1/3, 1/3, 1/3)$ is invariant; hence, clearly, the matrix is symmetric; $\kappa = \kappa(1,2)= \kappa(1,3)= \kappa(2,3)= 0.3$, $1-\kappa = 0.7$. 

We have, $V=0.7  * \hat {\cal P}$ with some stochastic matrix $\hat {\cal P}$. Hence, its spectral radius $r(V)$ equals 0.7. That is, our bound asymptotically coincides with the classical one (MD). 

Now compute the eigenvalues of $\cal P$ (without hat): $\lambda_1 = 1$, $\lambda_{2,3} \approx \pm 0.6082763$.  Recall that theoretically the eigenvalue/eigenvector approach  cannot be worse than any other method.  
\end{Ex}

\begin{Ex}[The new approach is similar to the MD; the eigenvalue one is better]
\[
{\cal P} = \left(
\begin{array}{l l l}
0 & 0.3 & 0.7  \\
0.7 & 0 & 0.3 \\
0.3 & 0.7 & 0
\end{array}
\right) \, \mbox{is doubly stochastic; so, clearly,} \; (1/3, 1/3,1/3) \; \mbox{is invariant.}
\]
Here $\kappa = \kappa(1,2)= \kappa(1,3)= \kappa(2,3)= 0.3$. 
So, $1-\kappa=0.7$. The operator $V$ has a form 
$$
V = 0.7 \times \hat {\cal P}, \quad \mbox{with some stochastic transition matrix $\hat {\cal P}$}. 
$$
It is known that such a matrix has its spectral radius $0.7$. Therefore, in this example our new bound is similar to the classical one MD.

One eigenvalue  of $\cal P$ equals 1. The characteristic equation on eigenvalues reads, 
\[
-\lambda^3 +3 \lambda * 0.7 * 0.3 + 0.7^3 +0.3^3 = 0 \, \Leftrightarrow \,  \lambda_1=1, \lambda_{2,3} = - 0.5 \pm \sqrt{0.25 - 0.37} = - 0.5 \pm i \sqrt{0.12}; 
\]
thus, 
\[
|\lambda_{2,3}| = \sqrt{0.25 + 0.12} = \sqrt{0.37} \approx 0.6082763 < 0.7.
\]
\end{Ex}

\begin{Ex}[The new approach is similar to the eigenvalue one and better than the MD one]
Consider the MP with the state space ${\cal S} = \{1,2,3\}$ and transition matrix
\[
{\cal P} = \left(
\begin{array}{l l l}
0 & 0.3 & 0.7  \\
1.0 & 0 & 0 \\
0.8 & 0.1 & 0.1
\end{array}
\right)\,.
\]
Here $\hat S^2 = \{(1,2), (1,3), (2,3), (2,1), (3,1), (3.2)\} =: (I,II, III, IV, V, VI)$, 
\[
\kappa(1,2) =\kappa(2,1) = 0,\;  \kappa(1,3) =\kappa(3,1) = 0.2,\; \kappa(2,3) =\kappa(3,2) = 0.8;
\]

Roots of the characteristic equation 

\begin{eqnarray*}
\lambda^{3}-0.1\lambda^{2}-0.86\lambda-0.04=0
\end{eqnarray*}
for the original transition probability matrix $\mathcal{P}$ are $\lambda_1=1$, $\lambda_2=-9/20 - \sqrt{65}/20$, $\lambda_3= -9/20 + \sqrt{65}/20$. 

The characteristic equation for the coupled process is
\begin{eqnarray*}
\lambda_{c}^{2}(0.02\lambda_{c}^{2} - 0.018\lambda_{c} + 0.0008)(0.02\lambda_{c}^{2} + 0.018\lambda_{c} + 0.0008)=0,
\end{eqnarray*}

and its roots are

$\lambda_{c}=-9/20 - \sqrt{65}/20$, $\lambda_{c}=-9/20 + \sqrt{65}/20$, a repeated root $\lambda_{c}=0$, $\lambda_{c}=-\sqrt{65}/20 + 9/20$, $\lambda_{c}=\sqrt{65}/20 + 9/20$, where the last root has the highest value, i.e. $r=\sqrt{65}/20 + 9/20$ and so $|\lambda_{2}|=r=|-9/20 - \sqrt{65}/20| \approx 0.85311289$ indeed.

In this example the classical MD approach is just useless (at least, for one step) since $1-\kappa=1$.

\end{Ex}

\begin{Ex}[The new approach is similar to the eigenvalue one and better than the MD one]
	Let us see the result for a symmetric 4 by 4 matrix:
	
	\[
	{\cal P} = \left(
	\begin{array}{l l l l}
	0.3 & 0.3 & 0.1 & 0.3  \\
	0.3 & 0.7 & 0 & 0 \\
	0.1 & 0 & 0.8 & 0.1 \\
	0.3 & 0 & 0.1 & 0.6 
	\end{array}
	\right)\,.
	\]
	
	The roots of the characteristic equation for the original transition probability matrix are $\lambda=-\sqrt{15}/10 + 2/5$,  $\lambda=3/5$, $\lambda=\sqrt{15}/10 + 2/5$ and  $\lambda=1$.
	
	The roots of the characteristic equation for the coupled process (which we do not show here) are
	
	$\lambda_{c}=0$ repeated twice, $\lambda_{c}=-\sqrt{15}/10 + 2/5$ repeated twice, $\lambda_{c}=-\sqrt{1993}/180 + 47/180$ repeated twice and $\lambda_{c}=\sqrt{1993}/180 + 47/180$ repeated twice, $\lambda_{c}=3/5$ repeated twice, $\lambda_{c}=\sqrt{15}/10 + 2/5$ repeated twice. As $r$ is the maximum of the eigenvalues, it is $r = \sqrt{15}/10 + 2/5$.
	
	We see that $|\lambda_{2}|=r=|\sqrt{15}/10 + 2/5| \approx 0.78729833$ indeed; $\kappa = \kappa(2,3) = 0.1$, $1-\kappa = 0.9$.
	
\end{Ex}

\begin{Ex}[The new approach is similar to the eigenvalue one and better than the MD one]
	Here is an example for a 4 by 4 matrix which uses the code for a general n by n matrix. 
	
		Consider the MP with the state space ${\cal S} = \{1,2,3,4\}$ and transition matrix
	\[
	{\cal P} = \left(
	\begin{array}{l l l l}
	0 & 0.2 & 0.3 & 0.5 \\
	0.4 & 0.3 & 0.3 & 0 \\
	0.8 & 0.1 & 0.1 & 0 \\
	0.5 & 0.3 & 0.2 & 0
	\end{array}
	\right)\,.
	\]

The characteristic equation for the original transition probability matrix is
\begin{eqnarray*}
(0.01*\lambda - 0.01)(1.0*\lambda^{3} + 0.6*\lambda^{2} + 0.03*\lambda + 0.01)=0
\end{eqnarray*}

And we have 

\begin{equation*}
-r = \lambda_{2}=-\frac{1}{3}\left(\frac{27\sqrt{73}}{1000}+\frac{27}{100}\right)^{1/3}-\frac{1}{5}-\frac{9}{100\left(\frac{27\sqrt{73}}{1000} + \frac{27}{100}\right)^{1/3}},
\end{equation*}
and $\lambda_2$ is a root of both characteristic equations: for the original matrix $\cal P$ and for $V=(1-\kappa)\hat{\mathcal{P}}$ with $\kappa = \kappa(1,3) = 0.2$, and  $(1-\kappa) = 0.8$.

The root of the characteristic equation for the coupled process with the maximum value is $r=\frac{1}{3}\left(\frac{27\sqrt{73}}{1000}+\frac{27}{100}\right)^{1/3}+\frac{1}{5}+\frac{9}{100\left(\frac{27\sqrt{73}}{1000} + \frac{27}{100}\right)^{1/3}}$.

So for this example indeed $r=|\lambda_{2}| \approx 0.57802908 < 1-\kappa$.

\end{Ex}

\begin{Ex}[The new approach is  worse than the eigenvalue one and  better than the MD one]
	\[
	{\cal P} = \left(
	\begin{array}{l l l l l}
	0 & 0 & 0 & 0.8 & 0.2 \\
	0 & 0.6 & 0.3 & 0.1 & 0  \\
	0 & 0.3 & 0.7 & 0 & 0 \\
	0.8 & 0.1 & 0 & 0.1 & 0 \\
	0.2 & 0 & 0 & 0 & 0.8  
	\end{array}
	\right)\,.
	\]
Here $\kappa = \kappa(2,5) = 0$, $1-\kappa = 1$, and the code shows  $r\approx 0.9354657>0.9324490 \approx  |\lambda_{2}|$.
	
\end{Ex}

\begin{Ex}[The new approach is  worse than the eigenvalue one and   better than the MD one]
	\[
	{\cal P} = \left(
	\begin{array}{l l l l l l}
	0 & 0.1 & 0.1 & 0.1 & 0.2 & 0.5 \\
	0.7 & 0.15 & 0.08 & 0.07 & 0 & 0 \\
	0.8 & 0.05 & 0.05 & 0.05 & 0.05 & 0 \\
	0.5 & 0.2 & 0.2 & 0.05 & 0.05 & 0 \\
	0.9 & 0 & 0.05 & 0 & 0.05 & 0 \\ 
	0 & 0 & 0.4 & 0.6 & 0 & 0 
	\end{array}
	\right)\,.
	\]
For this example $r\approx 0.695684 
>0.493945 \approx |\lambda_{2}|$; $\kappa(5,6)=0.05$, and $1-\kappa = 0.95$.
	
\end{Ex}

Computations have been performed for several ``randomly chosen'' matrices of large dimensions 
(up to $100 * 100$).  The results can be seen in the following table\footnote{The data and codes used and analysed in this study are available from the second author via GitHub \url{https://github.com/MariaVeretennikova/Markov-convergence/}
upon a reasonable request.}:\\

\begin{tabular}{|r|c|l|l|}
   \hline
    Matrix size & MD $(1-\kappa)$ & Proposed & $|\lambda_{2}|$ \\
  \hline
40 by 40 & 0.4648319 & 0.3329448 & 0.102591663 \\
    \hline
  50 by 50 & 0.7197204 & 0.5213607 & 0.18234762 \\
    \hline 
  70 by 70 & 0.6537795 & 0.5194678 & 0.15255626 \\
    \hline 
  90 by 90 &  0.6726586 & 0.518151 & 0.120973198\\
  \hline 
  100 by 100 & 0.648774 & 0.5155966 & 0.1212309 \\
  \hline
\end{tabular}

\medskip

It looks like in the general situation the new bound is better than the classical MD, while the spectral method gives the best bound. However, recall that the new bound, as well as the classical one are applicable to a much wider class of processes including non-homogeneous ones.

\medskip

Let us show one example related to the non-homogeneous MC. 
\begin{Ex}
A periodic case with period 2 is considered. A matrix ${\cal P}_1$ corresponds to all odd moments of time, whilst another transition matrix ${\cal P}_2$ corresponds to all even moments of time.

\[
{\cal P}_1 = \left(
\begin{array}{l l}
1/8 & 7/8 \\
1/3 & 2/3
\end{array}
\right)\,, 
\quad
{\cal P}_2 = \left(
\begin{array}{l l}
3/4 & 1/4 \\
2/3 & 1/3
\end{array}
\right)\,.
\]
Here are the calculation results.

\begin{itemize}
\item

The inequality (\ref{exp_bd3}), which is weakened by taking the minimum $\min_t \kappa_t$, that is, the maximum $\max_t (1 - \kappa_t)$ respectively, guarantees a bound with the asymptotics $0.25^n$.

\item
The inequality (\ref{exp_bd3}) with alternating multipliers $(1-\kappa_1)$ and  $(1-\kappa_2)$, gives an estimate with the asymptotics $0.1317616^n$.

\item
Using the spectral method for the product ${\cal P}_1*{\cal P}_2$ leads to the asymptotics $0.1317616^n$. 

\item
Applying bound (\ref{newrate13d}) gives the same asymptotics $0.1317616^n$.

\end{itemize}
The coincidence of the last three bounds is not surprising, because this case is two-dimensional. 

\end{Ex}


\begin{thebibliography}{99}




\bibitem{BV}
O.A. Butkovsky, A.Yu. Veretennikov, 
On asymptotics for Vaserstein coupling of Markov chains, Stochastic Processes and their Applications, 123(9), 2013, 3518-3541.

\bibitem{Stroock-Diaconis}
P. Diaconis and D. Stroock, Geometric bounds for eigenvalues of Markov chains, 
The Annals of Applied Probability, 1(1), 1991, 36-61.


\bibitem{Dobrushin}
R.L. Dobrushin, Central limit theorem for non-stationary Markov chains, I and II, Theory Prob. Appl., 1956, 1,  65-80, \& 329-383. 

\bibitem{Doeblin}
W. Doeblin, Expos\'e de la th\'eorie des chaines simples constantes de Markov \`a un nombre fini d'\'etats. Math\'ematique de l'Union Interbalkanique, 1938, 2, 77-105. 



\bibitem{Doob53}
J.L. Doob, Stochastic Processes, Wiley, 1953. 

\bibitem{Dynkin}
E.B. Dynkin, Markov processes: vol. 1, Springer, 2012. 


\bibitem{FW}
M.I. Freidlin, A.D. Wentzell, 
Random Perturbations of Dynamical Systems, Springer, New York et al., 1984.

\bibitem{Gantmacher}
F.R. Gantmacher,  The Theory of Matrices. vol. 1, 2,   AMS Chelsea Publ., Providence, Rhode Island, USA, 2000. DOI: https://doi.org/10.1126/science.131.3408.1216-a

\bibitem{Gnedenko}
B.V. Gnedenko, Theory of probability, 6th ed., Gordon and Breach Sci. Publ., Amsterdam, The Netherlands, et al., 1997.

\bibitem{Griffeath}
D.Z. Griffeath, A maximal coupling for Markov chains, 
Wahrscheinlichkeitstheorie verw. Gebiete,  1975, 31(2), 95-106.



\bibitem{Kalash}
V.V. Kalashnikov, Coupling Method, its Development and Applications, In the Russian translation of E. Nummelin,  General Irreducible Markov Chains and Non-negative Operators, 1984; Mir, Moscow, 1989,  176 - 190 (in Russian).
%


\bibitem{Karlin}
S. Karlin, H.M. Taylor, 
A first course in stochastic processes (2nd edition), 
Academic press, New York et al., 1975.


\bibitem{Koz}
V.S. Koziakin
On the computational aspects of the theory of joint spectral radius, Doklady Mathematics, 2009, 80(1), 487-491.

\bibitem{Kolm}
A.N. Kolmogorov, On the analytic methods of probability theory, Uspekhi Mat. Nauk, 1938, 5,  5-41 (in Russian).

\bibitem{KLS}
M.A. Krasnosel'skii,  E.A. Lifshits, A.V. Sobolev,  Positive Linear Systems: The method of positive operators, Berlin, Helderman Verlag, 1989. 

\bibitem{Lindvall}
T. Lindvall, Lectures on the Coupling Method, Dover Books on Mathematics, 2002.

\bibitem{Markov}
A.A. Markov, Extension of the law of large numbers to quantities dependent on each other, 
Izvestiia Fiz.-Matem. Obsch. Kazan Univ., (2nd Ser.), 15(1906), 135-156 (in Russian).

\bibitem{Markov2}
A.A. Markov, Extension of the law of large numbers to quantities dependent on each other, Selected  works on number theory and probability theory, Leningrad, 1951, 339 - 362 (in Russian).


\bibitem{Nummelin}
E. Nummelin, General Irreducible Markov Chains (Cambridge Tracts in Mathematics), CUP, Cambridge, 2008. 


\bibitem{Seneta1}
E. Seneta, Non-negative Matrices and Markov Chains, 2nd ed., Springer, New York et al, 1981.


\bibitem{Seneta3}
E. Seneta, Markov and the creation of Markov chains, 
Markov Anniversary Meeting 2006, https://www.maths.usyd.edu.au/u/eseneta/senetamcfinal.pdf



\bibitem{Thorisson}
H. Thorisson, Coupling, Stationarity, and Regeneration, Springer, New York, 2000.


\bibitem{Vaserstein}
L.N. Vaserstein, Markov Processes over Denumerable Products of Spaces, Describing Large Systems of Automata,  Problems of Information Transmission, 1969, 5(3), 47-52. 
 
\bibitem{Veretennikov17}
A.Yu. Veretennikov, Ergodic Markov processes and Poisson equations (lecture notes). In book: Modern problems of stochastic analysis and statistics - Selected contributions in honor of Valentin Konakov (editor: V.Panov). pp. 457 - 511. Springer, 2017.

\bibitem{VV2020}
A.Yu. Veretennikov, M.A. Veretennikova, On convergence rates for homogeneous Markov chains, Doklady Mathematics, 2020, 101(1),  12-15. 


\end{thebibliography}
\end{document}